\documentclass[a4paper,11pt]{article}
\usepackage{amsmath, amsfonts, amscd, amssymb, amsthm, enumerate, xypic, makeidx}

\def\bfB{\mathbf{B}}

\def\bfE{\mathbf{E}}
\def\bfF{\mathbf{F}}
\def\bfM{\mathbf{M}}

\DeclareMathOperator{\catch}{\operatorname{catch}}

\DeclareMathOperator{\id}{\operatorname{id}}

\DeclareMathOperator{\Mat}{\operatorname{M}}
\DeclareMathOperator{\Hom}{\operatorname{Hom}}

\def\calSH{\mathcal{S}\mathcal{H}}

\DeclareMathOperator{\Ker}{\operatorname{Ker}}

\DeclareMathOperator{\Mata}{\operatorname{A}}
\DeclareMathOperator{\Mats}{\operatorname{S}}
\DeclareMathOperator{\Math}{\operatorname{H}}
\DeclareMathOperator{\End}{\operatorname{End}}
\DeclareMathOperator{\Diag}{\operatorname{Diag}}
\DeclareMathOperator{\NT}{\operatorname{NT}}
\DeclareMathOperator{\GL}{\operatorname{GL}}
\DeclareMathOperator{\Vect}{\operatorname{span}}
\DeclareMathOperator{\Alt}{\operatorname{Alt}}
\DeclareMathOperator{\Lrad}{\operatorname{Lrad}}
\DeclareMathOperator{\Rrad}{\operatorname{Rrad}}

\DeclareMathOperator{\im}{\operatorname{Im}}

\DeclareMathOperator{\tr}{\operatorname{tr}}
\DeclareMathOperator{\maxrk}{\operatorname{maxrk}}
\DeclareMathOperator{\trk}{\operatorname{trk}}
\DeclareMathOperator{\sprk}{\operatorname{sprk}}
\DeclareMathOperator{\Tr}{\operatorname{Tr}}
\DeclareMathOperator{\car}{\operatorname{char}}

\DeclareMathOperator{\rk}{\operatorname{rk}}

\renewcommand{\setminus}{\smallsetminus}
\renewcommand{\epsilon}{\varepsilon}

\def\K{\mathbb{K}}
\def\F{\mathbb{F}}
\def\R{\mathbb{R}}
\def\C{\mathbb{C}}

\def\N{\mathbb{N}}

\def\H{\mathbb{H}}

\newcommand{\D}{\mathbb{D}}

\def\calA{\mathcal{A}}
\def\calB{\mathcal{B}}

\def\calD{\mathcal{D}}

\def\calH{\mathcal{H}}

\def\calL{\mathcal{L}}
\def\calM{\mathcal{M}}

\def\calS{\mathcal{S}}
\def\calT{\mathcal{T}}

\def\calV{\mathcal{V}}

\def\calX{\mathcal{X}}

\def\lcro{\mathopen{[\![}}
\def\rcro{\mathclose{]\!]}}

\theoremstyle{definition}
\newtheorem{Def}{Definition}[section]
\newtheorem{Not}[Def]{Notation}

\theoremstyle{plain}
\newtheorem{theo}{Theorem}[section]
\newtheorem{prop}[theo]{Proposition}
\newtheorem{cor}[theo]{Corollary}
\newtheorem{lemma}[theo]{Lemma}
\newtheorem{claim}{Claim}[section]

\theoremstyle{plain}

\theoremstyle{remark}
\newtheorem{Rems}{Remarks}[section]
\newtheorem{Rem}[Rems]{Remark}

\title{Affine subspaces of units in simple algebras}
\author{Cl\'ement de Seguins Pazzis\footnote{Universit\'e de Versailles Saint-Quentin-en-Yvelines, Laboratoire de Math\'ematiques
de Versailles, 45 avenue des Etats-Unis, 78035 Versailles cedex, France}
\footnote{e-mail address: clement.de-seguins-pazzis@ac-versailles.fr}}

\makeindex

\begin{document}

\thispagestyle{plain}

\maketitle

\begin{abstract}
Let $\calA$ be a simple algebra over a field $\F$.
Under a mild cardinality assumption on $\F$, we determine the greatest possible dimension for an $\F$-affine subspace of $\calA$
that is included in the group of units $\calA^\times$, and we describe the spaces that have the greatest possible dimension. This is equivalent to the problem of determining the greatest possible dimension for an $\F$-linear subspace $S$ of $\calA$ in which $x-1_\calA$ is a unit for all $x \in S$, and we elucidate the structure of these linear subspaces up to conjugation when their
dimension reaches the greatest possible one.

These classifications involve the associative composition algebras over $\F$. Over fields of characteristic other than $2$, 
the first problem is essentially reduced to the classification of nonisotropic quadratic forms over $\F$ and of nonisotropic Hermitian forms over quadratic and quaternionic extensions of $\F$.

These results are intimately connected with the problem of intransitive operator spaces
between finite-dimensional vector spaces over division rings, which we study in depth: in particular, 
we generalize a dual version of Atkinson's theorem on primitive spaces of bounded rank matrices.
\end{abstract}



\section{Introduction}

\subsection{The problem}

Let $\F$ be a field.
In the seminal \cite{Gerstenhaber}, Murray Gerstenhaber proved that
every linear subspace of nilpotent $n$-by-$n$ matrices (over $\F$) has its dimension no larger than $\frac{n(n-1)}{2}$,
and equality is attained only for the conjugates of the space $\NT_n(\F)$ of all strictly upper-triangular matrices
(Gerstenhaber actually required that $|\F| \geq n$, but his result has been later revealed to hold over an field \cite{Serezhkin}).

In \cite{Meshulamsymmetric}, Roy Meshulam noted that Gerstenhaber's theorem had a nice application to affine subspaces of invertible elements in matrix spaces:
if $\F$ is algebraically closed and $\calM$ is an affine subspace of $\Mat_n(\F)$ that contains only invertible matrices, 
then $\dim \calM \leq \frac{n(n-1)}{2}$ and equality is attained if and only if $\calM$ is equivalent to the 
space of all upper-triangular matrices with all diagonal entries equal to $1$.
The connection between the two problems is simple: every affine subspace of invertible matrices is equivalent to one which contains the identity matrix $I_n$,
and if an affine subspace $\calM$ contains $I_n$ then it consists only of invertible elements if and only if 
no element of the translation vector space $\overrightarrow{\calM}$ has a nonzero fixed vector: such a vector space is called a \textbf{trivial spectrum subspace}.
If the field is algebraically closed, the latter is equivalent to having $\overrightarrow{\calM}$ consist only of nilpotent matrices, hence the connection with 
Gerstenhaber's theorem for such fields. However, this is not true over more general fields, as for example the space $\Mata_n(\R)$ of all skew-symmetric real matrices 
is a trivial spectrum one but contains no nonzero nilpotent element. As a consequence $I_n+\Mata_n(\R)$ is a large affine subspace of invertible matrices,
but it is not difficult to prove that for $n>1$ it is not equivalent to the space $I_n+\NT_n(\F)$.

Meshulam's problem was solved in the early 2010's. In \cite{Quinlan} and \cite{dSPgivenrank}, the bound $\frac{n(n-1)}{2}$ from Gerstenhaber's theorem was generalized
to trivial spectrum subspaces over any field, and then a complete classification of the trivial spectrum subspaces with dimension $\frac{n(n-1)}{2}$, called the 
\emph{optimal} ones, was achieved \cite{dSPlargeaffinenonsingular} over any field with cardinality greater than $2$. Spectacularly, the classification is intimately connected with the \emph{quadratic} structure of the field $\F$: the classification of the optimal trivial spectrum subspaces is reduced to the classification 
of nonisotropic bilinear forms over finite-dimensional $\F$-vector spaces, up to congruence (i.e., equivalence and multiplication with a nonzero scalar).
In particular, for quadratically closed fields the trivial spectrum subspaces of maximal dimension all consist of nilpotent operators.
As a consequence, all the affine subspaces of invertible elements of $\Mat_n(\F)$ have dimension at most $\frac{n(n-1)}{2}$,
and (for fields with more than $2$ elements) the classification up to equivalence of the ones of maximal dimension, called the optimal ones, has been reduced to the classification of nonisotropic quadratic forms over $\F$ up to congruence. To be more precise, to each equivalence class of optimal affine subspace of invertible elements of $\Mat_n(\F)$ can be attached a list $(q_1,\dots,q_p)$ of nonisotropic quadratic forms on finite-dimensional vector spaces over $\F$
such that $n=\sum_{k=1}^p \dim q_k$, and the equivalence class of the said matrix space determines the congruence class of each $q_i$.
We must of course insist that this result fails for fields with $2$ elements.
For example, for $\F=\R$ this shows that there are as many equivalence classes of optimal affine spaces as there are ordered partitions of $n$.
Finally, these results have been applied to the study of affine subspaces of matrices with rank bounded below \cite{dSPlargeaffinerankbelow} and to the one of linear subspaces of diagonalisable matrices \cite{dSPSEVdiag2}.

Here we consider the more general problem of affine subspaces of units in simple algebras.
A simple algebra over $\F$ is a nontrivial unital associative $\F$-algebra $\calA$ with finite dimension over $\F$
and which has no nontrivial two-sided ideal. Two subsets $\calS$ and $\calT$ of $\calA$ are called \textbf{equivalent}
whenever there exist units $a,b$ in $\calA^\times$ such that $\calT=a \calS b$, in which case we write $\calS \sim \calT$, and they are called \textbf{similar}
whenever there exists a unit $a \in \calA^\times$ such that $\calT=a \calS a^{-1}$, in which case we write $\calS \simeq \calT$.

By a classical theorem of Wedderburn, the simple $\F$-algebras are the isomorphs of the matrix algebras of the form $\Mat_n(\D)$, where $\D$ is a division ring which contains $\F$ as a central subfield (i.e., $\F$ is included in the center of $\D$) and has finite dimension over $\F$, and $n$ is a positive integer.
Here, we expressly do not require that $\calA$ is central as an $\F$-algebra, for this would
leave out many interesting cases, e.g., the one of the $\R$-algebra $\Mat_n(\C)$.

Our main problem is the extension of Meshulam's: Given a simple $\F$-algebra $\calA$, what is the greatest possible dimension for an 
$\F$-affine subspace of $\calA$ that is included in the unit group $\calA^\times$? 
And what are the ones that have the greatest possible dimension?

We can easily translate Meshulam's observation in this general situation:
every $\F$-affine subspace of $\calA$ is equivalent to one which contains the unity $1_\calA$,
and hence we can largely bound the discussion by studying the $\F$-linear subspaces $S$ of $\calA$
for which $1_A-x$ is a unit for all $x \in S$.

\begin{Def}
A subset of $\calA$ is called \textbf{unital} when it is included in the group $\calA^\times$ of units of $\calA$.
A unital $\F$-affine subspace is called \textbf{optimal} when it has the greatest possible dimension among such subspaces.

An $\F$-linear subspace $S$ of $\calA$ is said to have \textbf{trivial spectrum} whenever 
$1_\calA+S$ is unital, i.e., $1_\calA-x \in \calA^\times$ for all $x \in S$.
It is called an \textbf{optimal trivial spectrum subspace} whenever it has the greatest possible dimension among such subspaces.
\end{Def}

We note that, at the other end of the spectrum, the analogue problem for the $\F$-affine subspaces of elements of $\calA$
that contain \emph{no} unit has been entirely solved in \cite{dSPFlandersskew}
(the greatest possible dimension has been obtained, as well as the structure of the spaces with greatest possible dimension).
It can hardly be disputed that this problem is considerably simpler than the one we are tackling here.

\subsection{Reduction to the irreducible spaces}\label{section:irreducibility}

In order to state our results more clearly, it is unavoidable to use Wedderburn's structure theorem for simple algebras and to frame
the problem into one on spaces of linear operators or of matrices. So, we now let $\D$ be an arbitrary division ring which contains $\F$ as a central subfield, and with finite dimension $d$ over $\F$. We will consider right vector spaces over $\D$ throughout, enabling us to represent $\D$-linear mappings between vector spaces by matrices with entries in $\D$ in the usual way (so that vector are represented by columns and applying an operator corresponds to multiplying the coordinate vector on the left by the matrix of the operator). For such vector spaces $U$ and $V$, we get the $\F$-vector space $\Hom_\D(U,V)$ of all $\D$-linear mappings from $U$ to $V$,
and the simple $\F$-algebra $\End_\D(U)$ of all endomorphisms of the $\D$-vector space $U$. The latter is isomorphic to the matrix $\F$-algebra $\Mat_n(\D)$
where $n:=\dim_\D U$ through a choice a basis of $U$. Throughout, we consider $\Mat_n(\D)$ as an $\F$-vector space and 
will not consider potential enriched structures (for example, when $\D=\C$ and $\F=\R$, we only view $\Mat_n(\C)$ as an $\R$-vector space, and never as 
a $\C$-vector space).

We can immediately give an example of a very large trivial spectrum subspace of $\Mat_n(\D)$, which generalizes Gerstenhaber's nilpotent spaces:
we can choose various $\F$-linear hyperplanes $H_1,\dots,H_n$ of $\D$ that do not contain $1_\D$, and form the 
set $H_1 \vee \cdots \vee H_n$ of all matrices of $\Mat_n(\D)$ of the form 
$$\begin{bmatrix}
h_1 & & [*] \\
& \ddots & \\
(0) & & h_n
\end{bmatrix}$$
where $h_i \in H_i$ for all $i \in \lcro 1,n\rcro$, and the entries above the diagonal are arbitrary.
It is clear that this is a trivial spectrum subspace of $\Mat_n(\D)$.
We note that this space has dimension 
$$\alpha(n,d):=n(d-1)+d\binom{n}{2}.$$
More generally, we can create new trivial spectrum subspaces of matrices from known ones thanks to the \emph{junction} operator $\vee$.
For nonempty subsets $\calM_1 \subseteq \Mat_{n_1}(\D),\dots,\calM_p \subseteq \Mat_{n_p}(\D)$, we define their \textbf{joint}
$\calM_1 \vee \cdots \vee \calM_p$ as the set of all block upper-triangular matrices of the form
$$\begin{bmatrix}
M_1 & & [*] \\
& \ddots & \\
[0] & & M_p
\end{bmatrix}$$
where $M_1 \in \calM_1,\dots,M_p \in \calM_p$ and the blocks above the diagonal are arbitrary.
Since $\D$ is finite-dimensional over $\F$, for any block upper-triangular matrix
$M=\begin{bmatrix}
A & B \\
[0] & C
\end{bmatrix}$ with entries in $\D$ the invertibility of $M$ is equivalent to the one of $A$ and $C$.
As a consequence, $\calM_1 \vee \cdots \vee \calM_p$ is a unital $\F$-affine subspace if and only if each $\calM_i$
is a unital $\F$-affine subspace, and  $\calM_1 \vee \cdots \vee \calM_p$ is a trivial spectrum linear subspace if and only if
each $\calM_i$ is a trivial spectrum linear subspace.

The above construction is of course intimately connected with the notion of irreducibility for subsets of endomorphisms.
Given a finite-dimensional vector space $V$ over $\D$ and a subset $\calX$ of $\End_\D(V)$,
a $\D$-linear subspace $V_0$ of $V$ is called \textbf{$\calX$-invariant} whenever $\forall f \in \calX, \; \forall x \in V_0, \; f(x) \in V_0$,
and we say that $\calX$ is \textbf{irreducible} whenever $V \neq \{0\}$ and the only $\calX$-invariant subspaces of $V$ are $\{0\}$ and $V$.
If $S$ is a trivial spectrum subspace and $V_0$ is an $S$-invariant subspace of $V$, then 
$S$ induces respective linear subspaces of $\Hom_\D(V_0)$ and $\Hom_\D(V/V_0)$, and both are trivial spectrum subspaces.

Then, we have the equivalent of the junction operator for subsets of endomorphisms:
given a flag $(V_0,\dots,V_p)$ of $\D$-linear subspaces of $V$, and given (non-empty) subsets
$\calS_1 \subseteq \End_\D(V_1/V_0),\dots,\calS_p \subseteq \End_\D(V_p/V_{p-1})$, we denote by 
$\calS_1 \vee \cdots \vee \calS_p$ the set of all $\D$-linear endomorphisms of $V$ that leave
each space $V_i$ invariant and induce, for each $i \in \lcro 1,p\rcro$, an endomorphism of $V_i/V_{i-1}$ that belongs to $\calS_i$.

We can now invoke a result that was obtained in the appendix of \cite{dSPtriangularizable} as a special case of a more general one:

\begin{theo}\label{theo:generalnonsense}
Let $S$ be an \emph{optimal} trivial spectrum subspace of $\End_\D(V)$ for some vector space $V$.
Then:
\begin{enumerate}[(a)]
\item The $S$-invariant subspaces of $V$ are totally ordered under inclusion, thereby forming a (potentially incomplete) flag $(V_0,\dots,V_p)$
of linear subspaces of $V$.
\item For all $i \in \lcro 1,p\rcro$, the $\F$-linear subspace $S_i$ of $\End_\D(V_i/V_{i-1})$
induced by $S$ is an irreducible optimal trivial spectrum subspace.
\item The space $S$ equals $S_1 \vee \cdots \vee S_p$.
\end{enumerate}
\end{theo}

In matrix terms, this translates as follows:

\begin{theo}
Let $\calM$ be an \emph{optimal} trivial spectrum subspace of $\Mat_n(\D)$ for some integer $n>0$.
Then there is a partition $n=n_1+\cdots+n_p$ into positive integers, and a list 
$(\calM_1,\dots,\calM_p)$ in which:
\begin{enumerate}[(i)]
\item For each $i \in \lcro 1,p\rcro$, the set $\calM_i$ is an irreducible optimal trivial spectrum subspace of $\Mat_{n_i}(\D)$.
\item The space $\calM$ is similar to $\calM_1 \vee \cdots \vee \calM_p$.
\item The similarity class of each $\calM_i$ is uniquely determined by the one of $\calM$.
\end{enumerate}
\end{theo}

This shows that the problem of classifying, up to conjugation, the optimal trivial spectrum subspaces of matrices over $\D$
can be reduced to the subproblem of classifying the irreducible ones.

One difficulty must be pointed out though, and we call it the \emph{reconstruction problem}. If we have irreducible optimal trivial spectrum subspaces of matrices
$\calM_1,\dots,\calM_p$, then although their joint $\calM_1 \vee \dots \vee \calM_p$ is surely a trivial spectrum subspace,
there can be no simple way to prove that it is optimal. To see this, consider the mundane case 
where we take $\F$-linear hyperplanes $H_1,\dots,H_n$ of $\D$ that do not contain $1_\D$,
and we form the joint $H_1 \vee \cdots \vee H_n$. Clearly each $H_i$ is an optimal trivial spectrum subspace of $\Mat_1(\D)$, 
and their joint can be optimal only if we prove that the greatest possible dimension for a trivial spectrum subspace of $\Mat_n(\D)$
is $\alpha(n,d)$. Yet, as we shall see this result is not known when $\F$ is an arbitrary finite field (unless $d=1$).

However, if we have an integer $n>0$ and we know that $\alpha(k,d)$ is the greatest possible dimension for a trivial spectrum subspace of 
$\Mat_k(\D)$ for all $k \in \lcro 1,n\rcro$ then for any partition $n=n_1+\cdots+n_p$ and any list
$(\calM_1,\dots,\calM_p)$ of optimal trivial spectrum subspaces of, respectively, $\Mat_{n_1}(\D),\dots,\Mat_{n_p}(\D)$, the space
$\calM_1 \vee \cdots \vee \calM_p$ is optimal since it is easily observed that 
$$n(d-1) +d\dbinom{n}{2}=\sum_{i=1}^p \left(n_i(d-1)+d\dbinom{n_i}{2}\right)+\sum_{1 \leq i<j \leq p} d n_i n_j.$$

\subsection{Review of older results}

Now that the main concepts are in place, we can review the known results on our problem.
As said in the introduction, these results deal with the special case $\D=\F$ only:

\begin{theo}[Quinlan \cite{Quinlan}, de Seguins Pazzis \cite{dSPgivenrank}]\label{theo:Quinlan}
Let $n>0$. The greatest possible dimension for a trivial spectrum subspace of $\Mat_n(\F)$ is $\dbinom{n}{2}$.
\end{theo}

The structure of the spaces of greatest possible dimension is connected with bilinear forms.
Let $V$ be an $\F$-vector space and $b$ be a bilinear form on $V$.
An endomorphism $u$ of $V$ is called \textbf{$b$-alternating} (on the right) 
whenever $\forall x\in V, \; b(x,u(x))=0$, i.e.\ the bilinear form $(x,y)\in V^2 \mapsto b(x,u(y))$ is alternating.

\begin{Not}
We denote by $\calA_b$ the set consisting of all the $b$-alternating endomorphisms of $V$.
This is a linear subspace of $\End_\F(V)$.
\end{Not}

Assume now that $b$ is nonisotropic, that is $\forall x \in V, \; b(x,x) \neq 0$. Then $\calA_b$ has dimension $\frac{n(n-1)}{2}$
(by the non-degeneracy of $b$), and it is a trivial spectrum subspace
(indeed, for all $u \in \calA_b$, having $u(x)=x$ for some $x \in V \setminus \{0\}$ implies that $0=b(x,u(x))=b(x,x)$).
It is not difficult to prove that it is also an irreducible one if $n>0$.

\begin{theo}[de Seguins Pazzis \cite{dSPlargeaffinenonsingular}]\label{theo:altoptimalirreducible}
Assume that $|\F|>2$.
Let $V$ be an $n$-dimensional vector space over $\F$, with $n>0$.
Then the irreducible optimal trivial spectrum subspaces of $\End_\F(V)$ are
the spaces of the form $\calA_b$ for some nonisotropic bilinear form $b$ on $V$.

For two such forms $b$ and $b'$, the spaces $\calA_b$ and $\calA_{b'}$ are equal if and only if 
$b'=\lambda b$ for some $\lambda \in \F^\times$, and they are conjugated in $\End_\F(V)$ if and only if $b'$ is congruent to $b$
(i.e., there exists an automorphism $\varphi$ of $V$ and a scalar $\alpha \in \F^\times$ such that $b'(x,y)=\alpha\, b(\varphi(x),\varphi(y))$
for all $x,y$ in $V$).
\end{theo}

As a consequence, we get the full picture for optimal trivial spectrum subspaces:

\begin{theo}
Assume that $|\F|>2$.
Let $V$ be an $n$-dimensional vector space over $\F$, with $n>0$.
Let $S$ be an optimal trivial spectrum subspace of $\End_\F(V)$, with flag of invariant subspaces denoted by $(V_0,\dots,V_p)$.
Then for each $i \in \lcro 1,p\rcro$ there is a nonisotropic bilinear form $b_i$ on $V_i/V_{i-1}$, uniquely determined by $S$
up to multiplication with a nonzero scalar, such that $S=\calA_{b_1} \vee \cdots \vee \calA_{b_p}$.
\end{theo}

Conversely, given a flag $(V_0,\dots,V_p)$ of linear subspaces of $V$, and for each $i \in \lcro 1,p\rcro$
a nonisotropic bilinear form $b_i$ on $V_i/V_{i-1}$, the set $S:=\calA_{b_1} \vee \cdots \vee \calA_{b_p}$
is an optimal trivial spectrum subspace of $\End_\F(V)$ with associated flag $(V_0,\dots,V_p)$ of invariant subspaces.

We can also state the matrix version of the problem, where a square matrix $P \in \Mat_n(\F)$
is called \textbf{nonisotropic} whenever $\forall X \in \F^n \setminus \{0\}, \; X^T PX \neq 0$,
and two square matrices $M$ and $M'$ are called \textbf{homocongruent} whenever there exist a scalar $\alpha \in \F^\times$
and an invertible matrix $Q\in \GL_n(\F)$ such that $M'=\alpha\, Q M Q^T$. Also, for $n \in \N^*$ we denote by $\Mata_n(\F)$
the space of all alternating $n$-by-$n$ matrices (i.e., skew-symmetric matrices with all diagonal entries zero).

\begin{theo}
Assume that $|\F|>2$.
\begin{enumerate}[(a)]
\item Let $n>0$ and $\calM$ be an optimal trivial spectrum subspace of $\Mat_n(\F)$. Then there exist a partition $n=n_1+\cdots+n_p$
and nonisotropic matrices
$P_1 \in \GL_{n_1}(\F),\dots,P_p \in \GL_{n_p}(\F)$ such that 
$$\calM \simeq (P_1^{-1} \Mata_{n_1}(\F)) \vee \cdots \vee (P_p^{-1} \Mata_{n_p}(\F)).$$
The partition $(n_1,\dots,n_p)$ is then uniquely determined by $n$, as well as the homocongruence classes of the matrices 
$P_1,\dots,P_p$.
\item Conversely, for each partition $n=n_1+\cdots+n_p$ and each list of 
 nonisotropic matrices $P_1 \in \GL_{n_1}(\F),\dots,P_p \in \GL_{n_p}(\F)$, the set
$(P_1^{-1} \Mata_{n_1}(\F)) \vee \cdots \vee (P_p^{-1} \Mata_{n_p}(\F))$ is an optimal trivial spectrum subspace of $\Mat_n(\F)$.
\end{enumerate}
\end{theo}

Finally, we can state the solution of the initial problem of classifying the optimal unital affine subspaces of $\Mat_n(\F)$.
Here, the matrix formulation turns out to be the clearer one:

\begin{theo}
Assume that $|\F|>2$.
\begin{enumerate}[(a)]
\item Let $n>0$ and $\calM$ be an optimal unital affine subspace of $\Mat_n(\F)$.
Then there exist a partition $n=n_1+\cdots+n_p$
and nonisotropic matrices
$P_1 \in \GL_{n_1}(\F),\dots,P_p \in \GL_{n_p}(\F)$ such that 
$$\calM \sim (P_1+ \Mata_{n_1}(\F)) \vee \cdots \vee (P_p+ \Mata_{n_p}(\F)).$$
The partition $(n_1,\dots,n_p)$ is then uniquely determined by $n$, as well as the congruence class of each quadratic form
$X \mapsto X^T P_i X$.
\item Conversely for each partition $n=n_1+\cdots+n_p$ and each list of 
 nonisotropic matrices $P_1 \in \GL_{n_1}(\F),\dots,P_p \in \GL_{n_p}(\F)$, the set
$(P_1+ \Mata_{n_1}(\F)) \vee \cdots \vee (P_p+ \Mata_{n_p}(\F))$ is an optimal unital affine subspace of $\Mat_n(\F)$.
\end{enumerate}
\end{theo}

Hence, this theorem gives a one-to-one correspondence between the equivalence classes of the optimal unital affine subspace of $\Mat_n(\F)$
on the one hand, and the lists of congruence classes of nonisotropic \emph{quadratic} forms of respective dimensions $n_1,\dots,n_p$
such that $n=\sum_{k=1}^p n_k$ on the other hand.

\begin{Rem}
In the previous classification results, the case of fields with two elements is systematically omitted. 
Indeed, it has been shown through various examples (see section 5.2 of \cite{dSPprimitiveF2}) that a classification theorem for $\F_2$ is hopeless.
\end{Rem}

\subsection{The new results on trivial spectrum subspaces}

We are now ready to state the results on the more general case of a simple $\F$-algebra.
From now on, we consider a division ring $\D$ with finite dimension $d$ over a central subfield $\F$.

We start with the statement on the greatest possible dimension for a trivial spectrum (linear) subspace.
This generalizes Theorem \ref{theo:Quinlan}, however with a cardinality assumption on $\F$.

\begin{theo}\label{theo:dimMax}
Let $\D$ be a division ring with dimension $d$ over a central subfield $\F$, and $V$ be a (right) vector space over $\D$ with finite dimension $n>0$. Assume that $|\F| \geq dn$.
Then the greatest possible dimension for a trivial spectrum subspace of $\End_\D(V)$ (or, equivalently, of $\Mat_n(\D)$) is
$$\alpha(n,d)=n(d-1)+d\,\dbinom{n}{2}.$$
\end{theo}

An immediate consequence is that when $|\F| \geq dn$ the reconstruction problem stated in the end of Section \ref{section:irreducibility}
has a positive answer in $\Mat_n(\D)$: Indeed in that case $|\F| \geq dk$ for all $k \in \lcro 1,n\rcro$, and for every partition
$n=n_1+\cdots+n_p$ the number $\sum_{k=1}^p \alpha(n_k,d)+d\sum_{1 \leq i<j \leq p} n_in_j$ equals $\alpha(n,d)$.

Next, we turn to the structure of the optimal spaces.
Here, we consider an $\F$-bilinear form $b : V^2 \rightarrow \F$ for a $\D$-vector space $V$, 
and set 
$$\calA_{b,\D}:=\calA_b \cap \End_\D(V)$$ 
(which is much smaller than the full space $\calA_b$ as we are only considering
$\D$-linear endomorphisms).

\begin{theo}[Partial classification theorem for irreducible optimal trivial spectrum subspaces]\label{theo:partialclassificationoptimal}
Let $V$ be a (right) $\D$-vector space with finite dimension $n>0$. Assume that $|\F|\geq nd$.
\begin{enumerate}[(i)]
\item If $n=1$ then the irreducible optimal trivial spectrum subspaces of $\End_\D(V)$
are the $\F$-linear hyperplanes of $\End_\D(V)$ that do not contain $\id_V$.
\item If $n \geq 2$ then every irreducible optimal trivial spectrum subspace of $\End_\D(V)$
equals $\calA_{b,\D}$ for some nonisotropic $\F$-bilinear form $b$ on $V$.
\end{enumerate}
\end{theo}

Note that point (i) is essentially obvious.

Of course, the above two theorems look on the surface as plain generalizations of the case $\D=\F$, and in some sense they are. But there is a subtlety:
although $\calA_{b,\D}$ is always a trivial spectrum subspace of $\End_\D(V)$, in general it is not an optimal one!
Hence it is necessary to understand when optimality happens.
Some additional terminology is necessary for this.

\begin{Def}
We say that $\D$ is of \textbf{quadratic type} over $\F$ whenever one of the following conditions holds:
\begin{enumerate}[(i)]
\item $d=1$;
\item $d=2$ and $\D$ is a separable extension of $\F$;
\item $d=4$ and $\D$ is a quaternion algebra over $\F$;
\item $d=2^n$ for some $n \geq 1$, $\car(\F)=2$ and $\D$ is a \textbf{hyper-radicial}\footnote{We prefer this terminology over the one of
purely inseparable extensions of exponent $1$, since it is more compact and quite natural.} extension of $\F$, i.e., $\D$ is a field and $\forall x \in \D, \; x^2 \in \F$.
\end{enumerate}
Cases (i) to (iii) are called the \textbf{separable types}.
\end{Def}

\begin{Def}\label{def:standardinvolution}
Assume that $\D$ is of quadratic type over $\F$.
Denote by $\sigma$ either the identity of $\F$ in cases (i) and (iv), the non-identity automorphism in case (ii),
and the quaternionic conjugation in case (iii); we say that $\sigma$ is the \textbf{standard involution on $\D$.}

\label{page:reducedtrace}
In cases (i) and (ii), we denote by $\tr_{\D/\F}$ the Galois trace of $\D$ over $\F$,
and in case (iii) we denote by $\tr_{\D/\F}$ the \emph{reduced} trace $x \mapsto x+\sigma(x)$ (the reduced trace of $x$ is either $2x$ if $x$ belongs to $\F 1_\D$,
otherwise it equals the opposite of the coefficient on $t$ of the minimal polynomial of $x$ over $\F$).

We say that $(\sigma,e)$ is an \textbf{associated pair} for $\D$ over $\F$ when $\D$ is of quadratic type over $\F$,
$\sigma$ is its standard involution, and $e$ is either $\tr_{\D/\F}$ when $\D$ is of separable type, or
an arbitrary nonzero $\F$-linear form on $\D$ otherwise.
\end{Def}

Finally, if $\sigma$ is an antiautomorphism of the $\F$-algebra $\D$, we define a $\sigma$-\textbf{sesquilinear} on
$V$ as a mapping $B : V^2 \rightarrow \D$ that is right-linear and left-$\sigma$-quasilinear, meaning that
$B(x\alpha+y,-)=\sigma(\alpha)B(x,-)+B(y,-)$ and $B(-,x\alpha  +y)=B(-,x)\alpha +B(-,y)$ for all $\alpha \in \D$ and $x,y$ in $V$.
Two such forms $B$ and $B'$ are called \textbf{equivalent} if there exists an automorphism $\varphi$ of the $\D$-vector space $V$
such that $B'(\varphi(x),\varphi(y))=B(x,y)$ for all $(x,y)\in V^2$.

\begin{theo}[Full classification theorem for optimal trivial spectrum subspaces]\label{theo:fullclassificationoptimal}
Let $V$ be a (right) $\D$-vector space with finite dimension $n>0$. Assume that $|\F|\geq nd$,
and let $b$ be a nonisotropic $\F$-bilinear form on $V$.
For $\calA_{b,\D}$ to be an irreducible optimal trivial spectrum subspace of $\End_\D(V)$,
it is necessary and sufficient that both of the following conditions hold:
\begin{enumerate}[(i)]
\item $\D$ is of quadratic type over $\F$;
\item Given an associated pair $(\sigma,e)$ for $\D$ over $\F$, there is a $\sigma$-sesquilinear form $B$ on the $\D$-vector space $V$ such that
$$\forall (x,y)\in V^2, \; b(x,y)=e(B(x,y)).$$
\end{enumerate}
In that case, $B$ is uniquely determined by $b$.
\end{theo}

Hence, if $\D$ is of quadratic type over $\F$ with associated pair $(\sigma,e)$,
and we take any $\sigma$-sesquilinear form $B$ on $V$, the space
$\calA_{b,\D}$ is an irreducible trivial spectrum subspace of $\End_\D(V)$ with dimension $\alpha(n,d)$
provided that $B$ is \textbf{$e$-nonisotropic}, 
meaning that $\forall x \in V, \; e(B(x,x))\neq 0$
(for this very result the provision $|\F| \geq nd$ is not needed, see the proof in Section \ref{section:alternator}).

As a consequence of Theorems \ref{theo:partialclassificationoptimal} and \ref{theo:fullclassificationoptimal}, and still under the cardinality assumption from Theorem
\ref{theo:dimMax}, we can refine our understanding of the irreducible optimal trivial spectrum subspaces:

\begin{cor}
Assume that $\D$ is not of quadratic type over $\F$, and with $|\F| \geq nd$.
The space $\Mat_n(\D)$ contains an irreducible optimal trivial spectrum subspace only if $n=1$.
\end{cor}

In case $\D$ is of quadratic type over $\F$, the following result slightly reinforces Theorem
\ref{theo:partialclassificationoptimal} by taking the $1$-dimensional spaces into account.

\begin{theo}\label{theo:partialclassificationoptimalwith1}
Assume that $\D$ is of quadratic type over $\F$.
Let $V$ be an $n$-dimensional vector space over $\D$, with $|\F| \geq nd$.
Then every irreducible optimal trivial spectrum subspace of $\End_\D(V)$
equals $\calA_{b,\D}$ for some nonisotropic $\F$-bilinear form $b$ on $V$.
Moreover, for two such forms $b$ and $b'$, the equality
$\calA_{b,\D}=\calA_{b',\D}$ is satisfied if and only if $b'$ is a scalar multiple of $b$.
\end{theo}

The next result deals with the similarity class of $\calA_{b,\D}$ inside the algebra $\End_\D(V)$.

\begin{theo}\label{theo:classuptosimilarity}
Let $V$ be a (right) $\D$-vector space with finite dimension $n>0$.
Assume that $\D$ is of quadratic type over $\F$, with associated pair $(\sigma,e)$.
Let $B$ and $B'$ be $e$-nonisotropic $\sigma$-sesquilinear forms on $V$, and set $b : (x,y) \mapsto e(B(x,y))$ and
$b' : (x,y) \mapsto e(B'(x,y))$.
Then $\calA_{b,\D}$ is similar to $\calA_{b',\D}$ in the algebra $\End_\D(V)$ if and only if
$B'$ is equivalent to $\alpha B$ as a $\sigma$-sesquilinear form for some $\alpha \in \F^\times$.
\end{theo}

It is a good idea now to express the above results in matrix terms. So, let us consider the space $V=\D^n$, with $n \geq 1$,
and let us identify $\End_\D(V)$ with the matrix space $\Mat_n(\D)$, so that the matrix $A$ corresponds to $X \in \D^n \mapsto AX \in \D^n$.
Assume that $\D$ is of quadratic type over $\F$, with associated pair $(\sigma,e)$.
For $A \in \Mat_{n,p}(\D)$, write $A^\star=(A^\sigma)^T$ as the $\sigma$-starconjugate of $A$ (so that $(AB)^\star=B^\star A^\star$
whenever $AB$ makes sense). The $\sigma$-sesquilinear forms on $\D^n$ are the mappings of the form
$$B_P : (X,Y) \mapsto X^\star P Y \quad \text{with $P \in \Mat_n(\D)$},$$
and such a form is nondegenerate if and only if $P$ is invertible. Of course, such a form is $e$-nonisotropic only if it is nondegenerate.
Now, letting $P \in \GL_n(\D)$ and setting $b : (X,Y) \in (\D^n)^2 \mapsto e(X^\star PY)$,
the alternating space $\calA_{b,\D}$ is the set of all $A \in \Mat_n(\D)$ such that $PA$ satisfies the following equation with unknown $M$:
\begin{equation}\label{eq:sigmalt}
\forall X \in \D^n, \; e(X^\star MX)=0.
\end{equation}

\begin{prop}\label{prop:skewHermitian}
The solution set of \eqref{eq:sigmalt} is the
$\F$-linear subspace
$$\calS\calH_n(\D):=\{M \in \Mat_n(\D) : M^\star=-M \quad \text{and} \quad \forall i \in \lcro 1,n\rcro, \; e(m_{i,i})=0\}$$
of all skew-Hermitian matrices whose diagonal entries belong to $\Ker e$.
\end{prop}

\begin{proof}Let $M \in \Mat_n(\D)$ satisfy \eqref{eq:sigmalt}.
The identity $\forall i \in \lcro 1,n\rcro, \; e(m_{i,i})=0$ is then obtained by applying \eqref{eq:sigmalt} to the vectors of the standard basis of $\D^n$.
Next, let $(X,Y)\in (\D^n)^2$. We find by polarizing that $e(X^\star MY)+e(Y^\star MX)=0$,
and hence
$$e(X^\star(M+M^\star) Y)=e(X^\star MY)+e((Y^\star M X)^\star)=e(X^\star MY)+e(Y^\star M X)=0.$$
Letting $a \in \D$ and applying this to $Ya$ instead of $Y$ leads to $X^\star(M+M^\star) Y=0$. Then, varying $X$ and $Y$ yields $M+M^\star=0$.

Conversely, let $M \in \calS\calH_n(\D)$. By the same computation as in the above, we see that the identity $M^\star=-M$
leads to $X \in \D^n \mapsto e(X^\star M X)$ being additive (its polar form vanishes).
As a consequence, for all $X=\begin{bmatrix}
x_1 & \cdots & x_n
\end{bmatrix}^T \in \D^n$, we find
$$e(X^\star M X)=\sum_{i=1}^n e(x_i^\star m_{i,i} x_i)=\sum_{i=1}^n e(\underbrace{x_ix_i^\star}_{\in \F} m_{i,i})=\sum_{i=1}^n x_ix_i^\star e(m_{i,i})=0.$$
\end{proof}

Next, we can give a clearer description of $\calS\calH_n(\D)$ by discussing the precise kind of pair $(\D,\F)$ under consideration.
\begin{itemize}
\item If $d=1$ then $\calS\calH_n(\D)=\Mata_n(\F)$, and the constraint $\forall i \in \lcro 1,n\rcro, \; m_{i,i}=0$
is critical to deal with fields with characteristic $2$.
\item If $d\geq 2$ and $\D$ is separable over $\F$ then $e=\id+\sigma$ and hence we simply have $\calS\calH_n(\D):=\{M \in \Mat_n(\D) : M^\star=-M\}$
in that case.
\item If $d \geq 2$ and $\D$ is inseparable over $\F$ then $\sigma=\id$ and
$\calS\calH_n(\D)$ is the space of all symmetric matrices of $\Mat_n(\D)$ whose diagonal entries belong to $\Ker e$ (in that case the constraint
$\forall i \in \lcro 1,n\rcro, \; e(m_{i,i})=0$ is unavoidable).
\end{itemize}

Finally, two matrices $A$ and $B$ of $\Mat_n(\D)$ will be called \textbf{homo-star-congruent} whenever 
$B=\alpha\,PAP^\star$ for some $\alpha \in \F^\times$ and some $P \in \GL_n(\D)$.

Hence, according to the type of $\D$ over $\F$, we can partially translate the previous results in matrix terms:

\begin{cor}
Assume that $\D$ is of quadratic type over $\F$, with associated pair $(\sigma,e)$. Assume that $|\F| \geq nd$.
Then:
\begin{enumerate}[(a)]
\item The irreducible optimal trivial spectrum subspaces of $\Mat_n(\D)$ are
the spaces of the form $P^{-1} \calS\calH_n(\D)$ where $P \in \GL_n(\D)$ is $e$-nonisotropic.
\item Given $e$-nonisotropic matrices $P_1$ and $P_2$ in $\GL_n(\D)$, the spaces $P_1^{-1}\calS\calH_n(\D)$ and $P_2^{-1} \calS\calH_n(\D)$ are
similar if and only if $P_2$ is homo-star-congruent to $P_1$.
\end{enumerate}
\end{cor}

At this point, the reader might be puzzled by the hyper-radicial case and doubt the existence of
such kinds of optimal trivial spectrum spaces. Yet, they actually exist in arbitrary dimensions. Take for instance, for $n \geq 2$ and $p \geq 1$,
a list $(X_1,\dots,X_{n+p})$ of algebraically independent indeterminates over $\F_2$, and consider
the field $\D=\F_2(X_1,\dots,X_{n+p})$ and the subfields
$$\F:=\F_2(X_1,\dots,X_n,(X_{n+1})^2,\dots,(X_{n+p})^2)$$
and
$$\K:=\F_2((X_1)^2,\dots,(X_n)^2,(X_{n+1})^2,\dots,(X_{n+p})^2).$$
Classically $[\D:\F]=2^p$, $[\F: \K]=2^n$
and $\K$ is the set of all squares of elements of $\D$; moreover $X_1,\dots,X_n$ are linearly independent over $\K$.
Now, take an arbitrary $\F$-linear form $e$ on $\D$ that maps $X_1,\dots,X_n$ to themselves,
and consider the diagonal matrix $P:=\Diag(X_1,\dots,X_n)$. For all $Y=\begin{bmatrix}
y_1 & \cdots & y_n
\end{bmatrix}^T \in \D^n$
we have $e(Y^TPY)=\sum_{i=1}^p (y_i)^2 X_i$, which vanishes if and only if $y_1=\cdots=y_n=0$.
Hence $P$ is $e$-nonisotropic. Therefore $P^{-1} \calS\calH_n(\D)$ is an irreducible optimal trivial spectrum subspace of $\Mat_n(\D)$, and
both $n$ and $p$ can be taken as large as we wish.

Let us conclude by combining Theorem \ref{theo:generalnonsense} with the previous results, so as to give a full description of the optimal
trivial spectrum spaces. We distinguish between two cases, whether $\D$ is of quadratic type over $\F$ or not.

\begin{theo}\label{theo:coro1}
Assume that $\D$ is not of quadratic type over $\F$.
Let $V$ be an $n$-dimensional vector space over $\D$, with $|\F| \geq nd$.
Let $S$ be an optimal trivial spectrum subspace of $\End_\D(V)$.
Then the $S$-invariant $\D$-linear subspaces form a complete flag $(V_0,\dots,V_n)$ of $V$,
and $S=H_1 \vee \cdots \vee H_n$, where for each $i \in \lcro 1,n\rcro$,  $H_i$ is an $\F$-linear hyperplane of $\End_\D(V_i/V_{i-1})$
that does not contain the identity of $V_i/V_{i-1}$.

Conversely, given a complete flag $(V_0,\dots,V_n)$ of $V$, and for all $i \in \lcro 1,n\rcro$ an $\F$-linear hyperplane $H_i$ of $V_i/V_{i-1}$
that does not contain the identity of $V_i/V_{i-1}$, the space $H_1 \vee \cdots \vee H_n$ is an optimal trivial spectrum subspace of $\End_\D(V)$.
\end{theo}

Here is a restatement in matrix terms:

\begin{theo}\label{theo:classifnonquadratic}
Let $n \geq 1$ be such that $|\F| \geq nd$, and assume that $\D$ is not of quadratic type over $\F$.

The optimal trivial spectrum subspaces of $\Mat_n(\D)$
are the spaces that are similar to a space of the form
$H_1 \vee \cdots \vee H_n$
where $H_1,\dots,H_n$ are $\F$-linear hyperplanes of $\D$ that do not contain $1_\D$.
Two such spaces $H_1 \vee \cdots \vee H_n$ and $H'_1 \vee \cdots \vee H'_n$ are similar in the algebra $\Mat_n(\D)$
if and only if, for all $i\in \lcro 1,n\rcro$, the hyperplanes $H_i$ and $H'_i$ are conjugated in $\D$.
\end{theo}

Hence, if $\D$ is not of quadratic type over $\F$ (and the cardinality of $\F$ is large enough) the classification of 
optimal trivial spectrum subspaces of $\Mat_n(\D)$ is reduced to the classification of the $\F$-linear hyperplanes of $\D$
that do not contain $1_\D$, up to conjugation. The latter classification is directly connected with the classification of 
conjugacy classes of elements in $\D$ by duality. To see this, consider the composite $\Tr : \D \rightarrow \F$
of the reduced trace of $\D$ over its center $C$ and of the Galois trace $\Tr_{C/\F}$: this is a nonzero $\F$-linear mapping that
satisfies $\Tr(xy)=\Tr(yx)$ for all $x,y$ in $\D$. Then $a \mapsto [b \mapsto \Tr(ab)]$ induces an $\F$-linear isomorphism
from $\D$ to the dual space of the $\F$-vector space $\D$. It is then easily checked that this yields a bijective correspondence 
between the $\F$-linear hyperplanes of $\D$ and the equivalence classes in $\D^\times$ for the relation $a \cong b$ defined by
$\exists \lambda \in \F^\times : \exists u \in \D^\times : b=\lambda\,u a u^{-1}$.
The $\F$-linear hyperplanes that do not contain $1_\D$ correspond to the elements of $\D$ with non-vanishing trace. Finally, by the Skolem-Noether theorem, 
two elements of $\D$ are conjugated if and only if they have the same minimal polynomial over the center $C$.

We finish with the case where $\D$ is of quadratic type over $\F$.

\begin{theo}\label{theo:coro3}
Assume that $\D$ is of quadratic type over $\F$, with associated pair $(\sigma,e)$.
Let $V$ be a $\D$-vector space with finite dimension $n>0$, where $|\F|\geq nd$, and $S$ be an optimal trivial spectrum subspace of $\End_\D(V)$.

Then the $S$-invariant $\D$-linear subspaces constitute a (potentially incomplete) flag $(V_0,V_1,\dots,V_p)$,
and for each $i \in \lcro 1,p\rcro$ there exists an $e$-nonisotropic $\sigma$-sesquilinear form
$B_i$ on $V_i/V_{i-1}$ such that, for $b_i:=e \circ B_i$,
$$S=\calA_{b_1,\D} \vee \cdots \vee \calA_{b_p,\D}.$$
Moreover, each form $B_i$ is uniquely determined by $S$ up to multiplication by an element of $\F^\times$,
and the equivalence class of $B_i$ is uniquely determined by the similarity class of $S$ in $\End_\D(V)$ up to multiplication by an element of $\F^\times$.

Conversely, given a (potentially incomplete) flag $(V_0,V_1,\dots,V_p)$, and for each
$i \in \lcro 1,p\rcro$ an $e$-nonisotropic $\sigma$-sesquilinear form
$B_i$ on $V_i/V_{i-1}$, with associated $\F$-bilinear form $b_i:=e \circ B_i$, the space
$\calA_{b_1,\D} \vee \cdots \vee \calA_{b_p,\D}$ is an optimal trivial spectrum subspace of $\End_\D(V)$.
\end{theo}

We finish with a restatement in matrix terms:

\begin{theo}\label{theo:coro4}
Assume that $\D$ is of quadratic type over $\F$, with associated pair $(\sigma,e)$.

Let $n \geq 1$ be such that $|\F| \geq nd$. For every optimal trivial spectrum subspace $\calM$ of $\Mat_n(\D)$, there exists
a partition $n=n_1+\cdots+n_p$ into positive integers and a list $(P_1,\dots,P_p) \in \GL_{n_1}(\D) \times \cdots \times \GL_{n_p}(\D)$
of $e$-nonisotropic matrices such that
$$\calM \simeq P_1^{-1}\calS\calH_{n_1}(\D) \vee \cdots \vee P_p^{-1}\calS\calH_{n_p}(\D).$$
The integers $n_1,\dots,n_p$ are uniquely determined by $n$, and for each $i \in \lcro 1,p\rcro$
the matrix $P_i$ is uniquely determined by $\calM$ up to homo-star-congruence.

Conversely, for any partition $n=n_1+\cdots+n_p$ and any list $(P_1,\dots,P_p) \in \GL_{n_1}(\D) \times \cdots \times \GL_{n_p}(\F)$
of $e$-nonisotropic matrices, the space $ P_1^{-1}\calS\calH_{n_1}(\D) \vee \cdots \vee P_p^{-1}\calS\calH_{n_p}(\D)$, as well as all its conjugates in $\Mat_n(\D)$,
is an optimal trivial spectrum subspace of $\Mat_n(\D)$.
\end{theo}

The corresponding results for the initial problem of optimal unital affine subspaces will be stated and proved in Section \ref{section:affinenonsingular}.

\subsection{The strategy, and intransitive operator spaces}\label{section:strategy}

Let us now explain the main strategy for proving our new results on trivial spectrum spaces.

So far, we have not succeeded in generalizing the techniques of \cite{dSPgivenrank,dSPlargeaffinenonsingular} to the case $d>1$.
This has led us to reconsider an alternative proof of Theorem \ref{theo:altoptimalirreducible}, limited to fields with more than $n-1$ elements,
which we discovered ten years ago \cite{dSPAtkinsontoGerstenhaber}. Critical to the proof is the observation that the trivial spectrum property implies a lack of
\emph{transitivity} from the said spaces of operators, meaning that if $S$ has trivial spectrum then $x\not\in S x$ for every nonzero vector $x \in V \setminus \{0\}$, and in particular $S x \subsetneq V$.
The key insight is then to introduce the \textbf{dual operator space} $\widehat{S}$ of $S$, which consists of all the evaluation mappings
$$\widehat{x} : u \in S \mapsto u(x) \in V, \quad \text{with $x \in V$.}$$
\label{pageref:dualoperatorspace}
Then $\widehat{S}$ is a linear subspace of $\Hom_\F(S,V)$, and every element of $\widehat{S}$ is nonsurjective and hence
has its rank less than $\dim_\F V$. The key in \cite{dSPAtkinsontoGerstenhaber} was to rely upon the theory of \emph{primitive} spaces of bounded rank matrices, studied first by Atkinson and Lloyd \cite{AtkLloydPrim,AtkinsonPrim} and rediscovered a few years later by Eisenbud and Harris \cite{EisenbudHarris} (see also the recent \cite{HuangLandsberg}). Thanks to Atkinson's classification of primitive matrix spaces that satisfy an extremal condition between the numbers of rows and columns (theorem C of \cite{AtkinsonPrim}), and more precisely to our own generalization to \emph{semi-primitive} spaces (see section 5 of \cite{dSPLLD2} and a restatement as theorem 3.1 in \cite{dSPAtkinsontoGerstenhaber}), it was then possible to find an easy derivation of Theorem \ref{theo:altoptimalirreducible} in case $|\F|\geq n$ (an unavoidable assumption if one wants to use Atkinson's theorem).

Yet, for the case $d>1$ a major issue is that the space $\widehat{S}$ is far from satisfying Atkinson's extremal condition between the respective
dimensions of the source and target spaces, even if $S$ is optimal; hence nothing can be obtained from the theory of primitive spaces of bounded rank matrices.
Nevertheless, we discovered that a version of Atkinson's theorem (theorem 5.3 of \cite{dSPLLD2}) can be entirely reformulated in the viewpoint of spaces of operators that lack transitivity, as we shall now see.

A few basic definitions and notation will help:

\begin{Not}\label{notation:StoV'}
Let $U$ and $V$ be $\D$-vector spaces and $V'$ be a $\D$-linear subspace of $V$.
Given a subset $S$ of $\Hom_\D(U,V)$, we denote by $S^{V'}:=S \cap \Hom_\D(U,V')$
the subset consisting of the elements of $S$ that map into $V'$.
\end{Not}

\begin{Def}
Let $U$ and $V$ be $\D$-vector spaces and $S$ be an $\F$-linear subspace of $\Hom_\D(U,V)$.
We say that $S$ is:
\begin{itemize}
\item \textbf{intransitive} whenever $\forall x \in U, \; S x \neq V$ ;
\item \textbf{deeply intransitive} whenever it is intransitive and, for every \emph{nonzero} linear subspace $V'$ of $V$, the set $S^{V'}$ is intransitive as an $\F$-linear subspace of $\Hom_\D(U,V')$.
\end{itemize}
In any case, the \textbf{transitive rank} $\trk(S)$ of $S$ is defined as the greatest possible dimension $\dim_\F(S x)$ when $x$ ranges over $U$.
\end{Def}
Note that $S$ is intransitive only if $V \neq \{0\}$, and in the definition of deep intransitivity we can replace
the condition that $S$ be intransitive with the one that $V \neq \{0\}$.

It is obvious that if $S$ is deeply intransitive then so is $S^{V'}$ for every nonzero $\D$-linear subspace $V'$ of $V$.
Constructing large deeply intransitive spaces can be achieved by using nonzero bilinear forms, as we shall now explain.

\begin{Def}
Let $U$ and $V$ be $\D$-vector spaces and $S$ be an $\F$-linear subspace of $\Hom_\D(U,V)$.
An $\F$-bilinear form $b : U \times V \rightarrow \F$ is called an \textbf{alternator} of $S$ whenever
$$\forall x \in U, \; \forall f \in S, \; b(x,f(x))=0.$$
The set of all such forms is denoted by $\Alt(S)$ (it is obviously a linear subspace of the set of all $\F$-bilinear forms on $U \times V$).
\end{Def}

\label{page:radicals}
For an $\F$-bilinear form $b : U \times V \rightarrow \F$, recall that the \textbf{left radical} of $b$ is the $\F$-linear subspace
$$\Lrad(b):=\{x \in U : \; b(x,-)=0\} \subseteq U,$$
and the \textbf{right radical} of $b$ is the $\F$-linear subspace
$$\Rrad(b):=\{y \in V : \; b(-,y)=0\} \subseteq V.$$
We say that $b$ is \textbf{right-nondegenerate} whenever $\Rrad(b)=\{0\}$.
Now, given $b$, we define $\calA_b$ as the set of all $\F$-linear mappings $f : U \rightarrow V$ such that
$\forall x \in U, \; b(x,f(x))=0$, and
$$\calA_{b,\D}:=\Hom_\D(U,V) \cap \calA_b$$
as the subset of all such mappings that are $\D$-linear.

We now make two key observations, which we gather in the next proposition.

\begin{prop}\label{prop:fromalternatortointransitive}
Let $U$ and $V$ be $\D$-vector spaces and $S$ be an $\F$-linear subspace of $\Hom_\D(U,V)$.
\begin{enumerate}[(a)]
\item If $S$ has a nonzero alternator then $S$ is intransitive.
\item If $S$ has a right-nondegenerate alternator and $V \neq \{0\}$ then $S$ is deeply intransitive.
\end{enumerate}
\end{prop}

\begin{proof}
Assume that $S$ has a nonzero alternator $b$. It follows by polarizing that
$\forall (x,x')\in U^2, \; \forall f \in S, \;  b(x,f(x'))+b(x',f(x))=0$, from which we observe that all the elements of $S$
map $\Lrad(b)$ into $\Rrad(b)$. The latter is not the full space $V$ because $b$ is nonzero.
Now, let $x \in U$. If $x \in \Lrad(b)$ then $S x \subseteq \Rrad(b) \subset V$.
Otherwise $S x$ is included in $\Ker b(x,-)$, an $\F$-linear hyperplane of $V$.
Hence $S$ is intransitive.

Assume now that $S$ has a right-nondegenerate alternator $b$ and $V \neq \{0\}$.
Let $V'$ be a nonzero $\D$-linear subspace of $V$. The restriction $b'$ of $b$ to $U \times V'$ is nonzero because
$b$ is right-nondegenerate, and obviously $b'$ is an alternator of $S^{V'}$. Hence $S^{V'}$ is intransitive, by point (a).
\end{proof}

Hence, whenever we take a right-nondegenerate bilinear form $b : U \times V \rightarrow \F$ and $V \neq \{0\}$,
every $\F$-linear subspace of $\calA_{b,\D}$ is deeply intransitive.
When $\D=\F$, Atkinson's theorem can be essentially summed up in saying that all large enough operator spaces are of this kind.
Let us now be more precise:

\begin{theo}[Atkinson's theorem for deeply intransitive operator spaces]
Let $U$ and $V$ be finite-dimensional $\F$-vector spaces and $S$ be a deeply intransitive linear subspace of $\Hom_\F(U,V)$.
Set $n:=\dim V$ and assume that $|\F| \geq n$.
Then:
\begin{enumerate}[(a)]
\item $\dim S \leq \frac{n(n-1)}{2}$, and even $\dim S \leq \frac{(n-1)(n-2)}{2}$ if $\trk(S)<n-1$.
\item If $\dim S \geq \frac{n(n-1)}{2}- (n-3)$ then $S$ has a right-nondegenerate alternator $b$,
and $\Alt(S)=\F b$.
\end{enumerate}
\end{theo}

As an extension of point (b), we note, with the same polarization argument as in the proof of Proposition \ref{prop:fromalternatortointransitive},
that all the operators in $S$ then vanish on the left radical $\Lrad(b)$, and hence
$S$ can naturally be identified with a linear subspace of $\Hom_\F(U/\Lrad(b),V)$,
for which the form induced by $b$ on $(U/\Lrad(b)) \times V$ is a nondegenerate alternator.
This fundamentally reduces the situation to the one where $\dim U=\dim V$, in which case we could go further and simply consider spaces of endomorphisms of $V$.

We will now make several comments on Atkinson's theorem. First of all, we must state that Atkinson's formulation of the theorem
has several additional assumptions that turn out to be useless.
Atkinson frames his result in the language of matrix spaces with bounded rank, and makes an assumption of \emph{primitivity}.
Primitivity consists of as many as four conditions that roughly mean that the rank constraint cannot be easily localized.
In contrast, here we replace all four assumptions in the definition of primitivity by a single assumption, which is deep intransitivity
(plain intransitivity is of course a far too weak hypothesis, allowing, e.g., $S$ to be the space of all operators from $U$ to a given linear hyperplane of $V$).
The deep intransitivity assumption is a dual counterpart of what Atkinson and Lloyd \cite{AtkLloydPrim} called the \emph{row property} for matrix spaces with bounded rank.
The row property for such a matrix space $\calM \subseteq \Mat_{n,p}(\F)$ means that, whenever we have invertible matrices
$P \in \GL_n(\F)$ and $Q \in \GL_p(\F)$, and integers $r \in \lcro 1,n\rcro$ and $s \in \lcro 1,p\rcro$ such that every matrix $M$ of $\calM$
looks like
$$M=P \begin{bmatrix}
[?]_{r \times s} & K(M) \\
[?]_{(n-r) \times s} & [0]_{(n-r) \times (p-s)}
\end{bmatrix} Q,$$
all the matrices in the extracted matrix space $K(\calM)$ have their rank less than $r$ (the number of rows).

Now, if we take an $\F$-linear subspace of $\Hom_\F(U,V)$, consider its dual operator space
$\widehat{S} \subseteq \Hom_\F(S,V)$, and represent the latter by a matrix space $\calM$ in bases of $S$ and $V$,
the row condition on $\calM$ is equivalent to the fact that whenever we take a nonzero linear subspace $S'$ of $S$
and we know that all the operators in $S'$ map into some nonzero $\F$-linear subspace $V'$ of $V$, then
the rank of $s' \in S' \mapsto s'(x) \in V'$ is less than the dimension of $V'$ for all $x \in U$, which is exactly saying that
$S'$ is intransitive as a linear subspace of mappings from $U$ to $V'$.
Hence, $\calM$ has the row property if and only if $S$ is deeply intransitive.

Among the four conditions that define primitivity under the Atkinson-Lloyd definition, two are worth mentioning:

\begin{Def}
An operator space $S \subseteq \Hom_\F(U,V)$ between $\F$-vector spaces $U$ and $V$
is called:
\begin{itemize}
\item \textbf{source-reduced} when $\underset{f \in S}{\bigcap} \Ker f=\{0\}$, i.e., no nonzero vector is annihilated by all the
operators in $S$;
\item \textbf{target-reduced} when $\underset{f \in S}{\sum} \im f=V$, i.e., no proper linear subspace of $V$ includes the range of all the operators in $S$.
\end{itemize}
Finally, $S$ is called \textbf{reduced} when it is both source-reduced and target-reduced.
\end{Def}

Atkinson's theorem is frequently formulated by adding the provision that the operator space $S$ under consideration is reduced
(see e.g., \cite{dSPLLD2}). However, this assumption is not necessary for the formulation in terms of intransitive spaces.
Indeed, on the one hand the dual operator space $\widehat{S}$ of an operator space $S \subseteq \Hom(U,V)$
is automatically source-reduced (as only the zero operator can be annihilated by all the evaluation mappings).
On the other hand, if $S$ is not target-reduced we can take the sum $V'$ of all the ranges of the operators in $S$,
then observe that $\widehat{S}$ can be viewed as a reduced subspace of $\Hom(S,V')$; applying point (a) in Atkinson's theorem then
readily shows that $\dim S \leq \frac{\dim V'(\dim V'-1)}{2}\cdot$

Atkinson's theorem deals with vector spaces over $\F$, but here we need to extend the situation to $\F$-linear subspaces of $\D$-linear mappings.
In that context, it makes no sense to try to extend Atkinson's formulation, as by representing the operators of $\widehat{S}$ by matrices with entries in $\F$,
we miss the critical fact that the operators in $S$ are $\D$-linear, and not only $\F$-linear. The correct way to generalize the result seems to be
as a theorem on intransitive operator spaces. We now state the generalization:

\begin{theo}[Main theorem on deeply intransitive operator spaces]\label{theo:deepintransitive}
Let $U$ and $V$ be finite-dimensional $\D$-vector spaces and $S$ be a deeply intransitive $\F$-linear subspace of $\Hom_\D(U,V)$.
Set $n:=\dim_\D V$ and assume that $|\F| \geq nd$.
Then:
\begin{enumerate}[(a)]
\item $\dim S \leq \alpha(n,d)$.
\item $\dim S \leq \alpha(n,d)-n+\max(2-d,0)$ if $\trk(S)<nd-1$.
\item If $\dim S \geq \alpha(n,d)-n+\max(4-d,2)$ then $S$ has a right-nondegenerate alternator $b$, and $\Alt(S)=\F b$.
\end{enumerate}
\end{theo}

\begin{Rem}\label{rem:n=2}
Due to point (a), the inequality $\dim S \geq \alpha(n,d)-n+\max(4-d,2)$ of point (c) requires that $n \geq 2$,
and even that $n \geq 3$ if $d=1$.
\end{Rem}

The next theorem further elucidates the structure of the alternators:

\begin{theo}\label{theo:deepintransitive2}
Let $U$ and $V$ be $\D$-vector spaces and $S$ be a deeply intransitive $\F$-linear subspace of $\Hom_\D(U,V)$.
Set $n:=\dim_\D V$, assume that $|\F| \geq nd$, that $S$ has a right-nondegenerate alternator $b$ and that
$\dim S \geq \alpha(n,d)-n+\max(4-d,2)$.

Then $\D$ has quadratic type over $\F$, and for any associated pair $(\sigma,e)$ the form
$b$ equals $(x,y) \mapsto e(B(x,y))$ for a unique right-nondegenerate $\sigma$-sesquilinear form $B$ on $U \times V$,
and of course $S \subseteq \calA_{b,\D}$.
\end{theo}

Here is a converse statement:

\begin{prop}\label{prop:deepintransitiveconverse}
Assume that $\D$ has quadratic type over $\F$, let $(\sigma,e)$ be an associated pair, and let
$B : U \times V \rightarrow \D$ be a right-nondegenerate $\sigma$-sesquilinear form.
Set $b : (x,y) \in U \times V \mapsto e(B(x,y))$. Then $\calA_{b,\D}$ has dimension $\alpha(n,d)$ (over $\F$), and it is deeply intransitive.
\end{prop}

Having these results, it will be fairly easy to derive the ones on optimal irreducible trivial spectrum subspaces,
by adapting the method of \cite{dSPAtkinsontoGerstenhaber}.
The main difficulty is that although a trivial spectrum subspace $S$ of $\End_\D(V)$ (with $V$ nonzero)
is obviously intransitive (since $S x$ can never contain $x$ when $x$ is a nonzero vector of $V$),
it is however not deeply intransitive in general.

A classical way of obtaining the deep intransitivity is through the following property, which is essentially the analogue of Atkinson and Lloyd's
condition (C) in the definition of a primitive space:

\begin{Def}
Let $U$ and $V$ be $\D$-vector spaces and $S$ be an $\F$-linear subspace of $\Hom_\D(U,V)$.
We say that $S$ is \textbf{weakly primitively intransitive} whenever
there is no $1$-dimensional $\D$-linear subspace $V'$ of $V$ such that,
for the standard projection $\pi : V \twoheadrightarrow V/V'$, the operator space $\pi S \subseteq \Hom_\D(U,V/V')$
satisfies $\trk(\pi S) \leq \tr(S)-d$.
\end{Def}

In fact, for most applications we will prefer the following stronger property, which is easier to work with:

\begin{Def}
Let $U$ and $V$ be $\D$-vector spaces and $S$ be an $\F$-linear subspace of $\Hom_\D(U,V)$.
We say that $S$ is \textbf{primitively intransitive} whenever it is intransitive and there is no nonzero linear subspace $V'$ of $V$ such that,
for the standard projection $\pi : V \twoheadrightarrow V/V'$, the operator space $\pi S \subseteq \Hom_\D(U,V/V')$ is intransitive.
\end{Def}

The following implication is easy to see:

\begin{prop}\label{prop:strongtoweakprimitive}
Every primitively intransitive operator space is weakly primitively intransitive.
\end{prop}

\begin{proof}
Let $U$ and $V$ be $\D$-vector spaces, $S$ be an $\F$-linear subspace of $\Hom_\D(U,V)$,
and assume that there exists a $1$-dimensional $\D$-linear subspace $V'$ of $V$ such that,
for the standard projection $\pi : V \twoheadrightarrow V/V'$, the operator space $\pi S \subseteq \Hom_\D(U,V/V')$
satisfies $\trk(\pi S)\leq \tr(S)-d$.
If $S$ is intransitive, then we deduce that $\trk(\pi S)<\dim_\F(V/V')$ and hence $\pi S$ is intransitive as well:
this shows that $S$ is not primitively intransitive.
\end{proof}

A key insight of Atkinson and Lloyd can be rephrased in saying that, provided that the field $\F$ has enough elements,
every weakly primitively intransitive space is deeply intransitive
(in their terminology, this is the statement that every matrix space that satisfies condition (C) has the row property). 
The primitive intransitivity condition has the shortcoming that it is generally not inherited
by $S^{V'}$ when $V'$ is a linear subspace of $V$, in contrast with the deep intransitivity condition. However, it is the most convenient condition
when applying the previous theorem. Indeed, we shall prove -- under the usual cardinality conditions on $\F$ --
that every irreducible optimal trivial spectrum subspace of $\End_\D(V)$
is primitively intransitive, whereas we are not aware of any direct proof that it must be deeply intransitive.

\subsection{Structure of the article}

As hinted at in the previous paragraph, our focus, which had already shifted from unital affine subspaces to 
trivial spectrum linear subspaces, is now shifting from trivial spectrum subspaces to deeply intransitive operator spaces,
a more profound and general problem. We expect that studying the later will yield many applications in the future,
just like Atkinson's theorem has found many unexpected applications in the recent past (see e.g.,
\cite{dSPAtkinsontoGerstenhaber,dSPtriangularizable}).
In an early version of the present article, we wrote a proof of the result on trivial spectrum subspaces alone.
Surprising as it may seem, we found later that the result on deeply intransitive operator spaces was much easier to prove. The icing on the cake is that it
also holds far more potential for future applications.

Let us now discuss the structure of the remainder of the article.
We will start with a small general section that regroups the basic facts we need on division algebras, including the classification of multiplicative
quadratic forms on such algebras. Then we will immediately give a proof of the dimension theorem for trivial spectrum subspaces
(Theorem \ref{theo:dimMax}), along with two key technical lemmas. We could of course have waited for the proofs of the theorems on intransitive operator spaces,
but the proof is so short that the reader will certainly enjoy reading it immediately.

The next section (Section \ref{section:genericmatrices}) introduces the key technical tools to study the intransitive operator spaces:
there we make no claims of originality, as we will use the generic matrix machinery that allowed Atkinson to prove his theorem,
a machinery which has already demonstrated its remarkable efficiency for such problems.
The main results on intransitive operator spaces are obtained in Section \ref{section:bigproof}: there,
the technical details for a general division ring $\D$ are much more difficult than in Atkinson's article.

In Section \ref{section:trivialspectrum} we derive all the main results on optimal trivial spectrum subspaces from the ones on intransitive operator spaces.

Finally, in Section \ref{section:affinenonsingular} we conclude the study by solving the initial problem of classifying the optimal unital $\F$-affine subspaces
of a simple $\F$-algebra, up to equivalence.

\subsection{Potential applications}

There are many applications of the results featured in this article, but due to a lack of space we have decided to keep them out of this article.
However, we can quickly discuss two of them.

The first application is towards affine subspaces of operators with rank bounded below.
In this, one takes finite-dimensional $\D$-vector spaces $U$ and $V$, together with a positive integer $r \leq \min(\dim_\D U,\dim_\D V)$,
and one asks about the greatest possible dimension for an $\F$-affine subspace of elements of $\Hom_\D(U,V)$
in which every operator has rank at least $r$, and how to classify the spaces of greatest possible dimension. 
The special case $r=\dim_\D U=\dim_\D V$ is essentially the problem of unital affine subspaces in the simple $\F$-algebra $\End_\D(U)$. 
Deriving the greatest possible dimension in the general case is not difficult, but classifying the optimal spaces is highly non-trivial:
see \cite{dSPlargeaffinerankbelow} for the special case where $\D=\F$.

A second potential application is towards $\F$-linear subspaces of $\F$-diagonalisable elements of a simple $\F$-algebra $\calA$.
There, an element $a \in \calA$ is called $\F$-diagonalisable whenever its minimal polynomial over $\F$ splits with simple roots, 
or equivalently when $x \mapsto ax$ is a diagonalisable endomorphism of the $\F$-vector space $\calA$.
When $\calA=\Mat_n(\F)$ this problem has already been studied, in particular over the field of real numbers:
for an arbitrary field $\F$, any $\F$-linear subspace of $\F$-diagonalisable elements of $\Mat_n(\F)$
has dimension at most $\frac{n(n+1)}{2}$, with equality attained only for conjugates of the space $\Mats_n(\F)$ of all symmetric matrices
(and if $n \geq 2$ such spaces exist only if $\F$ is formally real and Pythagorean).
A connection with the problem of trivial spectrum subspace of $\Mat_n(\F)$ was discovered and exploited in \cite{dSPSEVdiag2}:
In a sequel to this article we will use this connection to generalize the study of linear subspaces of $\F$-diagonalisable elements in 
any simple $\F$-algebra, and in particular we will obtain the structure of the 
$\R$-linear subspaces of $\R$-diagonalisable matrices of $\Mat_n(\C)$: it will turn out that the optimal spaces are the conjugates of the space
$\Math_n(\C)$ of all Hermitian matrices. There is also a quaternionic analogue of this result.

\section{On $\D$-vector spaces seen as $\F$-vector spaces}\label{section:basicD}

From now on, $\D$ denotes a division algebra that contains $\F$ as a central subfield, and we denote by $d$ the dimension of $\D$ over $\F$.

\subsection{The socle of an $\F$-linear subspace}\label{section:socle}

\begin{Def}\label{def:socle}
Let $U$ be a $\D$-vector space, and $U_0$ be an $\F$-linear subspace of it.
Then there is a greatest $\D$-linear subspace of $U$ that is included in $U_0$: we call it the \textbf{socle} of $U_0$.
\end{Def}

It suffices to take the sum of all $\D$-linear subspaces of $U$ that are included in $U$: it is obviously included in $U$, and it is a $\D$-linear subspace.

\subsection{$\D$-linear forms vs $\F$-linear forms}\label{section:duality}

Now, we fix a nonzero linear form $e$ on the $\F$-vector space $\D$.
Consider an arbitrary finite-dimensional $\D$-vector space $V$.
We have a homomorphism of $\F$-vector spaces
$$\Phi_e : \begin{cases}
\Hom_\D(V,\D) & \longrightarrow \Hom_\F(V,\F) \\
f & \longmapsto e\circ f,
\end{cases}$$
and we will see that it is actually an isomorphism.
First, let us see that $\Phi_e$ is injective: for all $f \in \Hom_\D(V,\D) \setminus \{0\}$, the mapping
$f$ is surjective and hence cannot map into $\Ker e$, therefore $e \circ f \neq 0$.
We conclude that $\Phi_e$ is an isomorphism because both its source and target spaces have
dimension  $d \dim_\D V$ over $\F$.

Here is a first consequence which will be frequently used:

\begin{lemma}\label{lemma:hyperplane}
Let $V$ be a finite-dimensional $\D$-vector space, and $H$ be an $\F$-linear hyperplane of it.
Then the socle of $H$ is a $\D$-linear hyperplane of $V$.
\end{lemma}

\begin{proof}
Choose a nonzero $\F$-linear form $\varphi$ with kernel $H$. Then $\varphi=e \circ f$ for a unique nonzero $\D$-linear form $f$ on $V$.
Putting $G:=\Ker f$, we see that $G$ is a $\D$-linear hyperplane and $G \subseteq H$.
The socle $S$ of $H$ is a strict subspace of $V$ (since $H \neq V$), so $\dim_\D S \leq \dim_\D V-1$, and since $G \subseteq S$ we deduce that $G=S$.
\end{proof}

A second consequence lies in the representation of $\F$-bilinear forms, to be used in the study of alternators.
Let $b : U \times V \rightarrow \F$ be an $\F$-bilinear form, where $U$ and $V$ are finite-dimensional $\D$-vector spaces. We can then consider the associated $\F$-linear mapping
$$x \in U \mapsto b(x,-) \in \Hom_\F(V,\F),$$
and then compose it with the inverse of $\Phi_e$ to obtain an $\F$-linear mapping
$$\widetilde{b} : U \longrightarrow \Hom_\D(V,\D),$$
yielding an $\F$-bilinear and right-$\D$-linear mapping
$$b^{\D,e} : (x,y)\in U \times V \mapsto \widetilde{b}(x)[y],$$
so that
\label{page:inducedform}
$$\forall (x,y)\in U \times V, \; b(x,y)=e(b^{\D,e}(x,y)).$$
Moreover $b^{\D,e}$ is uniquely determined by $b$, and it depends $\F$-linearly on $b$.

Here is a simple result in the special case where $b^{\D,e}$ is sesquilinear for some involution $\sigma$ of $\D$ over $\F$.

\begin{lemma}\label{lemma:relationradicals}
Assume that $\D$ has quadratic type over $\F$, and consider an associated pair $(\sigma,e)$. Let $U$ and $V$ be finite-dimensional $\D$-vector space,
$B : U \times V \rightarrow \D$ be a $\sigma$-sesquilinear form, and set $b:=e \circ B$. Then
the left radical (respectively, the right radical) of $b$ coincides with the one of $B$, and in particular it is a $\D$-linear subspace of $U$
(respectively, of $V$).
\end{lemma}

\begin{proof}
Let $y \in \Rrad(b)$. Then
$$\forall x \in U,\; \forall a \in \D, \; e(aB(x,y))=e(B(x \sigma(a),y))=b(x\sigma(a),y)=0$$
and hence $B(-,y)=0$. Hence $\Rrad(b)$ is included in the right radical of $B$; the converse inclusion is obvious. Likewise one proves that
the left radical of $B$ is the left radical of $b$.
\end{proof}

\subsection{Multiplicative quadratic forms on division algebras}

Our study will require the following classical result from the theory of composition algebras.

\begin{theo}\label{theo:compositionalgebras}
Let $q$ be a nonisotropic quadratic form on the $\F$-vector space $\D$ such that $\forall (x,y)\in \D^2, \; q(xy)=q(x)q(y)$.
Then $\D$ is of quadratic type over $\F$, and $q : x \mapsto x\sigma(x)$ for the standard involution $\sigma$ of $\D$ over $\F$.
\end{theo}

The result is rarely presented in this way in the literature, in particular because most references on the topic do not restrict the scope to nonisotropic quadratic forms.
Hence, some comments are necessary. If $q$ is non-degenerate then the theory of unital composition algebras yields that either $\D$ is a Cayley algebra over $\F$ or 
$\D$ is of separable quadratic type over $\F$ and $q : x \mapsto x\sigma(x)$ for the standard involution $\sigma$ of $\D$ over $\F$
(see theorem 33.17 in \cite{Involutions}). Yet, here the multiplication map is associative, so Cayley algebras are discarded.
Finally, if $q$ is degenerate then $\car(\F)=2$ and it is known in that case that the radical of $q$ equals the full space $\D$
(see proposition 33.4 in \cite{Involutions}), and from there it is known that $\D$ is a hyper-radicial extension of $\F$ and that $q : x \mapsto x^2$
(see corollary 33.6 in \cite{Involutions}).

\section{The greatest possible dimension for a trivial spectrum subspace}\label{section:majodim}

This short section is devoted to the proof of Theorem \ref{theo:dimMax}. We have seen right at the beginning of 
Section \ref{section:irreducibility} that for every integer $n \geq 1$
there exists a trivial spectrum subspace of $\Mat_n(\D)$ with dimension $\alpha(n,d)=n(d-1)+d\dbinom{n}{2}$.
Hence, in order to prove Theorem \ref{theo:dimMax} it suffices to prove the following result:

\begin{theo}\label{theo:majodimsimple}
Let $\D$ be a division ring with finite dimension $d$ over a central subfield $\F$, and $V$ be a vector space over $\D$ with finite dimension $n$. Assume that $|\F| \geq dn$.
Let $S$ be a trivial spectrum subspace of $\End_\D(V)$.
Then
$$\dim S \leq \alpha(n,d)=n(d-1)+d\dbinom{n}{2}.$$
\end{theo}

The main key to this result is the use of the dual operator space of $S$ (see Section \ref{section:strategy}) and what we call
the \emph{Tangent Space Lemma.}
The latter is recalled in the next section.

\subsection{The Tangent Space Lemma, and the Flanders-Atkinson lemma}

One of the main keys to Theorem \ref{theo:majodimsimple} is the following lemma (see \cite{AtkinsonPrim} and \cite{dSPLLD1} for various proofs).
The special case $k=0$ in the second set of identities is due to Flanders \cite{Flanders}, and the general case is due to Atkinson.

\begin{lemma}[Flanders-Atkinson lemma]\label{lemma:FlandersAtkinson}
Let $n,p,r$ be integers with $0<r \leq \min(n,p)$. Assume that $|\F|>r$.
Let $J_r:=\begin{bmatrix}
I_r & 0 \\
0 & 0
\end{bmatrix}$ and $M=\begin{bmatrix}
A & C \\
B & D
\end{bmatrix}$ belong to $\Mat_{n,p}(\F)$, with $A \in \Mat_r(\F)$ and so on.
Assume that $\rk(sJ_r+t M) \leq r$ for all $(s,t) \in \F^2$. Then
$$D=0 \quad \text{and} \quad \forall k \geq 0, \; B A^kC=0.$$
\end{lemma}

The following corollary is simply obtained by exploiting the identity $D=0$ in the Flanders-Atkinson lemma:

\begin{lemma}[Tangent Space Lemma]\label{lemma:tangentspace}
Let $U$ and $V$ be finite-dimensional vector spaces over $\F$, and $S$ be a linear subspace of $\Hom_\F(U,V)$.
In $S$, take an element $u_0$ of maximal rank $r$, and assume that $|\F|>r$.
Then every element of $S$ maps $\Ker u_0$ into $\im u_0$.
\end{lemma}

The reference to tangent spaces comes from the observation that if $\F$ is the field of real numbers then
the manifold of all rank $r$ operators of $\Hom_\F(U,V)$ has its tangent space at $u_0$
equal to the space of all operators in $\Hom_\F(U,V)$ that map $\Ker u_0$ into $\im u_0$.

\subsection{Proof of Theorem \ref{theo:majodimsimple}}

We prove the result by induction on $n$. The case $n=0$ is trivial, so we assume that $n \geq 1$.
Let us consider the dual operator space
$\widehat{S}$, consisting of the $\F$-linear mappings
\label{page:hatx}
$$\widehat{x} : u \in S \mapsto u(x) \in V, \quad \text{with $x \in V$.}$$
Note that $\widehat{S}$ is an $\F$-linear subspace of $\Hom_\F(S,V)$.
Let $x \in V \setminus \{0\}$. Since $S$ has trivial spectrum, there is no $u \in S$ such that $u(x)=x$; hence $\rk(\widehat{x}) \leq dn-1$.

Let us now take $x \in V \setminus \{0\}$ such that $\widehat{x}$ has the greatest possible rank $r \leq dn-1$ in $\widehat{S}$.
Since $|\F|>r$, the Tangent Space Lemma (Lemma \ref{lemma:tangentspace}) yields
$$\forall u \in \Ker \widehat{x}, \; \forall y \in V, \quad \widehat{y}(u) \in \im \widehat{x},$$
or in other words
$$\forall u \in S, \; u(x)=0 \Rightarrow \im u \subseteq S x.$$
Setting
$$S':=\Ker \widehat{x}=\{u \in S : \; u(x)=0\},$$
we can apply the rank theorem to get
\begin{equation}\label{eq1:thrang}
\dim(S)=\dim(S')+\dim_\F (S x) \leq \dim(S')+(dn-1).
\end{equation}
Now, the operators in $S'$ are actually $\D$-linear so their ranges are $\D$-linear subspaces of $V$,
all included in $S x$ and hence all included in the socle (see Definition \ref{def:socle}) of $S x$.
This socle does not contain $x$ (because $x \not\in S x$), and hence we can embed it into a
$\D$-hyperplane $H$ of $V$ such that $x\D \oplus H=V$, and we note that
$\im u \subseteq H$ for every $u \in S'$.

Hence, all the elements of $S'$ leave $H$ invariant, inducing $\D$-linear endomorphisms of $H$, all of them with trivial spectrum.
Since $x\D \oplus H=V$ and the elements of $S'$ vanish on $x\D$, the $\F$-linear mapping
$$u \in S' \mapsto u_{H} \in \End_\D(H)$$
is injective. Its range is a trivial spectrum subspace $\calT$ of $\End_\D(H)$.
Since $|\F| \geq (n-1)d$, we obtain by induction that
$$\dim(S') =\dim (\calT) \leq d \dbinom{n-1}{2}+(n-1)(d-1).$$
By combining this with \eqref{eq1:thrang}, we conclude that
$$\dim(S) \leq  (n-1)(d-1)+d \dbinom{n-1}{2}+((d-1)+d(n-1))=n(d-1)+d\dbinom{n}{2},$$
which completes the proof.

\subsection{Adaptation to deeply intransitive operator spaces}

We shall immediately prove the first part of Theorem \ref{theo:deepintransitive} by using essentially the same technique as in the previous proof.
We restate the result below:

\begin{prop}\label{prop:deepintransitivepointa}
Let $U$ and $V$ be $\D$-vector spaces, with $n=\dim_\D V \geq 1$. Assume that $|\F| \geq nd$.
Let $S$ be a deeply intransitive $\F$-linear subspace of $\Hom_\D(U,V)$.
Then $\dim S \leq \alpha(n,d)$.
\end{prop}

\begin{proof}
We denote by $r$ the transitive rank of $S$. If $r=0$ then $S=\{0\}$ and we are done.
Assume that $r>0$, and choose $x \in U$ such that $\dim_\F S x=r$.
Hence $\widehat{x}$ has the greatest possible rank in $\widehat{S}$, and as  $|\F|>r$
we can apply the tangent space lemma. This yields that all the operators in
$S':=\{f \in S : f(x)=0\}$ map into $S x$.
As in the previous proof, we can find a $\D$-linear hyperplane $H$ of $V$
into which all the elements of $S'$ map.
The rank theorem yields
$$\dim S=r+\dim S'.$$
Then there are two cases:
\begin{itemize}
\item If $n=1$ then $S'=\{0\}$ and we conclude that $\dim S \leq d-1$, as claimed;
\item If $n>1$ then $S'$ is a deeply intransitive $\F$-linear subspace of $\Hom_\D(U,H)$, by induction we find that
$$\dim S \leq (nd-1)+ \alpha(n-1,d)=\alpha(n,d),$$
and the proof is complete.
\end{itemize}
\end{proof}

Point (b) of Theorem \ref{theo:deepintransitive} is much more difficult when $d>1$, and we wait until Section \ref{section:bigproof} to prove it.
The next point is an easy consequence of Proposition \ref{prop:deepintransitivepointa}.

\begin{cor}\label{cor:sufficientconditiontargetreduced}
Let $S$ be a deeply intransitive $\F$-linear subspace of $\Hom_\D(U,V)$ such that
$\dim S \geq \alpha(n,d)-n+\max(3-d,1)$ and $n>0$.
Then $S$ is target-reduced.
\end{cor}

\begin{Rem}\label{remark:Dtargetreduced}
In the proof and in future use of the result, it is crucial to note that the fact that $S$ is not target-reduced
is equivalent to the existence of a $\D$-linear hyperplane $H$ of $V$ (and not simply of an $\F$-linear hyperplane) that includes the range of every element of $S$.
The converse implication is obvious. For the direct one, assume that we have
an $\F$-linear hyperplane $V'$ of $V$ that includes the range of every operator in $S$.
Then all the operators in $S$ map into the socle of $V'$, which is a $\D$-linear hyperplane of $V$.
\end{Rem}

\begin{proof}[Proof of Corollary \ref{cor:sufficientconditiontargetreduced}]
Assume that $S$ is not target-reduced. Then there exists a $\D$-linear hyperplane $H$ of $V$ such that $S \subseteq \Hom_\D(U,H)$.

Because $S$ is deeply intransitive we can view it as a deeply intransitive subspace of $\Hom_\D(U,H)$, and then
Proposition \ref{prop:deepintransitivepointa} yields $\dim S \leq \alpha(n-1,d) =\alpha(n,d)-(nd-1)$.
Now $(nd-1)-n=n(d-1)-1$, and $n(d-1)-1 >-2$ if $d=1$, whereas $n(d-1) -1>-1$ whenever $d \geq 2$ and $n>0$.
Hence, in any case we contradict the asumption that $\dim S \geq \alpha(n,d)-n+\max(3-d,1)$.
\end{proof}

\section{A review of the method of generic matrices}\label{section:genericmatrices}

Here we review the technique of generic matrices that was introduced by Atkinson and Lloyd
\cite{AtkLloydPrim,AtkinsonPrim} (see also \cite{dSPLLD2}).

\subsection{Generic matrices}

Let $E$ be an $\F$-vector space. We consider its dual vector space $E^\star:=\Hom_\F(E,\F)$
and the (graded) symmetric algebra $R=S(E^\star)$ over $\F$ (which is the direct sum of the symmetric tensor spaces $S^n(E^\star)$ for $n \geq 0$)
and its field of fractions $\mathbb{L}$. We naturally view $\F$ as a subfield of $\mathbb{L}$.

The $1$-homogeneous elements of $R$ constitute the space $S^1(E^\star)=E^\star$ of symmetric $1$-tensors.
We have a canonical ring homomorphism $S(E^\star) \rightarrow \mathrm{Pol}(E,\F)$ to the ring of polynomial functions on $E$
(it is an isomorphism if $\F$ is infinite). Hence for every $\mathbf{p} \in S(E^\star)$ and every $z \in E$ we can consider the specialization
$\mathbf{p}(z) \in \F$. More generally, given a matrix $\mathbf{M} \in \Mat_{m,p}(R)$ and a vector $z \in E$, we denote by
$\mathbf{M}(z)$ the matrix obtained by specializing all the entries of $\mathbf{M}$ at $z$.

\begin{Def}
Let $m,p$ be positive integers, and let $\calM$ be an $\F$-linear subspace of $\Mat_{m,p}(\F)$.
A \textbf{generic matrix}\footnote{Beware that in \cite{dSPLLD1}, such a matrix is called a semi-generic matrix.} for $\calM$ (with respect to $R$) consists of a matrix
$\mathbf{M}$ whose entries are $1$-homogeneous elements of $R$ and
such that
$$\calM=\{\mathbf{M}(z) \mid z \in E\}.$$
\end{Def}

\begin{Rem}
In \cite{AtkLloydPrim,AtkinsonPrim,dSPLLD2}, the authors only dealt with polynomial rings that are formally of the form $\F[\mathbf{t_1},\dots,\mathbf{t_r}]$.
Our viewpoint here will turn out to be more convenient, although it is not fundamentally more general than the traditional one.
Indeed, if we have a basis $(e_1,\dots,e_r)$ of the $\F$-vector space $E$, we have an isomorphism of $\F$-algebras from
$\F[\mathbf{t_1},\dots,\mathbf{t_r}]$ to $S(E^\star)$ that maps each indeterminate $\mathbf{t_i}$ to the $1$-symmetric tensor $e_i^\star$, the $i$-th vector of the dual basis
of $(e_1,\dots,e_r)$.
\end{Rem}

A critical property of generic matrices, when seen as matrices with entries in the fraction field $\mathbb{L}$,
is that their rank coincides with the maximal rank of $\calM$ when the field $\F$ is large enough.

\begin{Not}\label{not:maxrank}
For a nonempty subset $\calM$ of $\Mat_{m,p}(\F)$, we set
$$\maxrk(\calM):=\max \{\rk M \mid M \in \Mat_{m,p}(\F)\}.$$
\end{Not}

The following classical lemma will be used later, and it also has a direct application to generic matrices:

\begin{lemma}\label{prop:spanrank}
Let $\F'$ be an arbitrary field extension of $\F$, and $\calM$ be an $\F$-linear subspace of $\Mat_{m,p}(\F)$.
Assume that $|\F|>\maxrk(\calM)$. Then
$$\maxrk(\Vect_{\F'}(\calM))=\maxrk(\calM).$$
\end{lemma}

By working with $\F'=\mathbb{L}$, it is easy to derive the following corollary:

\begin{cor}[See lemma 2.1 in \cite{dSPLLD1}]\label{cor:genericrank}
Let $\calM$ be an $\F$-linear subspace of $\Mat_{m,p}(\F)$, and $\mathbf{M}$ be a generic matrix of it.
Assume that $|\F|>\maxrk(\calM)$.
Then
$$\rk_\mathbb{L} (\mathbf{M})=\maxrk(\calM).$$
\end{cor}

In connection with this corollary, we remark that the generic matrix $\mathbf{M}$ belongs to $\Vect_{\mathbb{L}}(\calM)$, a fact that we will reuse later.
To see this, we take a basis $(e_1,\dots,e_p)$ of $E$, consider the dual basis $(e_1^\star,\dots,e_p^\star)$,
and note that $\mathbf{M}=\sum_{k=1}^p e_k^\star\, \mathbf{M}(e_k)$.

\subsection{Notions of ranks for generic row matrices}

Let us now discuss row matrices over the ring $R$.

\begin{Def}\label{def:spanrank}
Let $\mathbf{L} \in \Mat_{1,n}(R)$ be a row matrix whose entries are $1$-homogeneous elements of $R$.
The \textbf{spanning rank} of $\mathbf{L}$, denoted by $\sprk (\mathbf{L})$, is defined as the dimension (over $\F$) of the specialized
linear subspace $\{\mathbf{L}(z) \mid z \in E \}$ of $\Mat_{1,n}(\F)$.
\end{Def}

Of course, one should not confuse the spanning rank of $\mathbf{L}$ with its (traditional) rank, i.e., its rank as a matrix of $\Mat_{1,n}(\mathbb{L})$,
which equals $0$ or $1$ (whereas the spanning rank is an integer between $0$ and $n$).

The following lemma (lemma 5.2 in \cite{dSPLLD2}) will be particularly important in our study:

\begin{lemma}[Factorization Lemma]\label{lemma:genericcollinearity}
Let $n \geq 1$ and $\delta \geq 1$, let $\mathbf{X}$ and $\mathbf{Y}$ be two vectors of $\Mat_{1,n}(R)$ and assume that:
\begin{enumerate}[(i)]
\item The entries of $\mathbf{X}$ are $1$-homogeneous and the ones of $\mathbf{Y}$ are $\delta$-homogeneous.
\item The spanning rank of $\mathbf{X}$ is greater than $1$.
\item The vectors $\mathbf{X}$ and $\mathbf{Y}$ are linearly dependent in the $\mathbb{L}$-vector space $\Mat_{1,n}(\mathbb{L})$.
\end{enumerate}
Then $\mathbf{Y}=\mathbf{p}\,\mathbf{X}$ for some $(\delta-1)$-homogeneous $\mathbf{p} \in R$.
\end{lemma}

\subsection{Recognizing alternators in generic matrix identities}

We finish this introduction to generic matrices by discussing left annihilators of generic matrices and their relationship with
alternators.

\begin{Def}\label{def:catchers}
Let $\mathbf{M} \in \Mat_{m,p}(R)$. A \textbf{left-annihilator} of $\mathbf{M}$ is a row matrix $\mathbf{L} \in \Mat_{1,m}(R)$
such that that $\mathbf{L}\mathbf{M}=0$. A \textbf{catcher} of $\mathbf{M}$ is a left-annihilator of $\mathbf{M}$
whose entries are $1$-homogeneous.
We denote by $\catch(\mathbf{M})$ the $\F$-linear subspace of $\Mat_{1,m}(R)$ that consists of all the catchers of $\mathbf{M}$,
and call it the \textbf{catcher space} of $\mathbf{M}$.
\end{Def}

Catchers are intimately connected with alternators, as we shall now see.

Let us take finite-dimensional vector spaces $U$ and $V$ over $\F$, a linear subspace $S$ of $\Hom_\F(U,V)$,
and let $\bfB_S=(v_1,\dots,v_n)$ and $\bfB_V=(e_1,\dots,e_p)$ be respective bases of $S$ and $V$.

\emph{Throughout, we assume that $S$ is target-reduced} (see the definition in Section \ref{section:strategy}),
to the effect that $\widehat{S}$ is also target-reduced.
We consider the  matrix space $\calM$ that represents the dual operator space $\widehat{S}$ in $\bfB_S$ and $\bfB_V$, and
we take an arbitrary generic matrix $\mathbf{M}$ that is associated with $\calM$, with parameter space denoted by $E$.
We define the following linear mappings:
\begin{itemize}
\item $\Phi$ assigns to every $z \in E$ the operator in $\widehat{S}$ whose matrix in the bases
$\bfB_S$ and $\bfB_V$ is $\mathbf{M}(z)$;
\item $\widehat{(-)}$ assigns to every $y \in U$ the evaluation mapping $\widehat{y} \in \widehat{S}$.
\end{itemize}
These mappings are represented in the following diagram
$$\xymatrix{ & E \ar@{>>}[d]^\Phi  \\
U \ar@{>>}[r]^{\widehat{(-)}} & \widehat{S}.
}$$
Now, we consider an alternator $b$ of $S$. The left mapping $L_b : y \in U \mapsto b(y,-) \in V^\star$
can then be inserted into the following diagram:
$$\xymatrix{ & E \ar@{>>}[d]^\Phi & \\
U \ar@{>>}[r]^{\widehat{(-)}} \ar@/_1pc/[rr]_{L_b} & \widehat{S}  & V^\star.
}$$
Next, we prove that $L_b$ factors into the composite of $\widehat{(-)}$ with a unique linear mapping $\Theta : \widehat{S} \rightarrow V^\star$.
To see this, we apply the factorization lemma for linear maps: all we need is to check that
$L_b$ vanishes at every vector $y \in U$ such that $\widehat{y}=0$.
So, let $y \in U$ satisfy $\widehat{y}=0$. Hence, for all $f \in S$, we have $b(y,f(z))=-b(z,f(y))=0$ because $f$ is $b$-alternating.
Because $S$ is target-reduced, this yields $y \in \Ker L_b$, as claimed.

The previous diagram is then completed as:
$$\xymatrix{ & E \ar@{>>}[d]^\Phi \ar@{-->}[dr]^{\Psi} & \\
U \ar@{>>}[r]^{\widehat{(-)}} \ar@/_1pc/[rr]_{L_b} & \widehat{S} \ar@{-->}[r]^{\Theta}  & V^\star,
}$$
where $\Psi:=\Theta \circ \Phi$. Finally, we consider the matrix
$\mathbf{L} \in \Mat_{1,p}(R)$ whose entries are the composite linear forms $z \in E \mapsto \Psi[z](e_i)$, with $i \in \lcro 1,p\rcro$.
We claim that $\mathbf{L} \mathbf{M}=0$.
Since the entries of  $\mathbf{L} \mathbf{M}$ are homogeneous of degree $2$ in $R$, it will suffice
to prove that $\forall z \in E, \; \mathbf{L}(z) \mathbf{M}(z)=0$.
Letting $z \in E$ and choosing $y \in U$ such that $\Phi(z)=\widehat{y}$, we see from the definition that, for all $j \in \lcro 1,n\rcro$, the $j$-th entry of
$\mathbf{L}(z) \mathbf{M}(z)$ equals $\sum_{i=1}^p \Psi[z](e_i)\, e_i^\star(v_j(y))$, where $(e_1^\star,\dots,e_p^\star)$
stands for the dual basis of $(e_1,\dots,e_p)$. Hence
$$\forall j \in \lcro 1,n\rcro, \; \bigl(\mathbf{L}(z) \mathbf{M}(z)\bigr)_j=\Psi[z][v_j(y)]=\Theta(\widehat{y})[v_j(y)]=b(y,v_j(y))=0.$$
Moreover, if $\mathbf{L}=0$ then we successively find $\Psi=0$, $\Theta=0$ and finally $b=0$.
Assigning $\mathbf{L}$ to $b$ thus creates an injective linear mapping
$$\Lambda : \Alt(S) \hookrightarrow \catch(\mathbf{M}).$$

Conversely, let $\mathbf{L}$ be a catcher of $\mathbf{M}$.
Let $\Psi : E \rightarrow V^\star$ assign to every $z \in E$ the linear form
on $V$ whose matrix in the basis $\bfB_V$ is $\mathbf{L}(z)$.
We claim that there is a (unique) factorization $\Psi=\Theta \circ \Phi$ for some linear mapping $\Theta : \widehat{S} \rightarrow V^\star$.
Again, we use the factorization lemma. Let $z_0 \in \Ker \Phi$.
Let us polarize the quadratic identity
$\mathbf{L}\mathbf{M}=0$ at $z_0$. Because $\mathbf{M}(z_0)=0$ by assumption, this yields
$\mathbf{L}(z_0)M=0$ for all $M \in \calM$. In other words, every operator in $\widehat{S}$ maps into $\Ker \Psi(z_0)$.
Since $\widehat{S}$ is target-reduced, we deduce that $\Psi(z_0)=0$. Hence the existence of $\Theta$, the uniqueness of which is obvious because $\Phi$ is onto.
Finally, we consider the composite mapping
$$y \in U \longmapsto \Theta(\widehat{y}) \in V^\star$$
and the associated bilinear form
$$b : (y,z)\in U \times V \longmapsto \Theta(\widehat{y})[z].$$
Then, by the same computation as in the above we deduce from the identity
$\mathbf{L}\mathbf{M}=0$ that $b$ is an alternator of $\calV$.
And finally it is clear that $\Lambda(b)=\mathbf{L}$.

We conclude that $\Lambda : \Alt(\calV) \overset{\simeq}{\longrightarrow} \catch(\mathbf{M})$ is a vector space isomorphism.

\subsection{Simple applications to deeply intransitive operator spaces}

We immediately give a simple application of the previous considerations to the
structure of the alternator space of an intransitive operator space.

\begin{prop}\label{prop:quasitransitive}
Let $U$ and $V$ be finite-dimensional vector spaces over $\D$, with $n:=\dim_\D V$.
Let $S$ be a target-reduced $\F$-linear subspace of $\Hom_\D(U,V)$.
Assume that $\trk S=nd-1$ and $|\F| \geq nd$. Then $\dim \Alt(S) \leq 1$.
\end{prop}

\begin{proof}
We choose respective bases $\bfB_V=(e_1,\dots,e_{nd})$ and $\bfB_S=(f_1,\dots,f_N)$ of $V$ and $S$ seen as $\F$-vector spaces,
and we consider the generic matrix $\mathbf{M}$ associated with $x \in U \mapsto \widehat{x} \in \widehat{S}$
in these bases.
The condition that $\trk S=nd-1$ yields $\rk \mathbf{M}=nd-1$, by Corollary \ref{cor:genericrank}.

Let $b_1$ and $b_2$ be nonzero alternators of $S$. We shall prove that $b_1$ and $b_2$ are linearly dependent over $\F$.

We apply the procedure explained in the previous paragraph to obtain
associated catchers $\mathbf{B}_1$ and $\mathbf{B}_2$ of $\mathbf{M}$.
Since $\rk \mathbf{M}=nd-1$, these catchers are linearly dependent over the field $\mathbb{L}$ of fractions of the
ring $R$ of symmetric tensors of the dual space $U^\star$.

Now, assume that $\sprk \mathbf{B}_1=1$. Then $\mathbf{B_1}=\mathbf{b_1} B_1$ for some nonzero row $B_1 \in \Mat_{1,nd}(\F) \setminus \{0\}$
and some $1$-homogeneous element $\mathbf{b_1}$ of $R \setminus \{0\}$.
Then the identity $\mathbf{B_1} \mathbf{M}=0$ leads to $\mathbf{b_1}(B_1 \mathbf{M})=0$, and hence to $B_1 \mathbf{M}=0$.
The matrix $B_1$ represents a linear form $\varphi$ in the dual basis of $\mathbf{B}_V$, and the statement that $B_1 \mathbf{M}=0$
means that all the operators in $\widehat{S}$ map into the kernel of $\varphi$. Hence so do all the operators in $S$,
and we contradict the assumption that $S$ is target-reduced.

Hence $\sprk \mathbf{B}_1>1$ and we deduce from the Factorization Lemma that $\mathbf{B_2}=\alpha \mathbf{B}_1$
for some scalar $\alpha \in \F$. This yields $b_2=\alpha b_1$. We conclude that $\Alt(S)$ has dimension at most $1$.
\end{proof}

\begin{prop}\label{prop:primitiveintransitivitytodeep}
Let $U$ and $V$ be finite-dimensional vector spaces over $\D$, with $n:=\dim_\D V$.
Let $S$ be an $\F$-linear subspace of $\Hom_\D(U,V)$, and assume that $|\F| \geq nd$.
\begin{enumerate}[(a)]
\item If $S$ is weakly primitively intransitive, then it is deeply intransitive.
\item If $S$ is primitively intransitive, then it is deeply intransitive.
\end{enumerate}
\end{prop}

\begin{proof}
Assume that $S$ is not deeply intransitive. Then we can find
a nonzero $\D$-linear subspace $V'$ of $V$ such that $S^{V'}$ is not intransitive.
Denote by $\pi : V \twoheadrightarrow V/V'$ the standard projection, and consider the operator
space $\pi S \subseteq \Hom_\D(U,V/V')$.

Take a basis $(v_1,\dots,v_s)$ of $S^{V'}$; and
lift a basis of $\pi S$ to a family $(u_1,\dots,u_r)$ in $S$, so that $(u_1,\dots,u_r,v_1,\dots,v_s)$ is a basis $\bfB_S$ of $S$.
Take an arbitrary $\F$-basis $\bfB_V$ of $V$ that is adapted to $V'$
and in which the first $d$ vectors are pairwise linearly dependent over $\D$, and consider a generic matrix $\mathbf{M}$ of $\widehat{S}$
in the bases $\bfB_S$ and $\bfB_V$. Then this matrix takes the form
$$\mathbf{M}=\begin{bmatrix}
[?]_{p \times r} & \mathbf{C} \\
\mathbf{B} & [0]
\end{bmatrix} \quad \text{where $p:=\dim_\F(V')$.}$$
Note that $\mathbf{C}$ is a generic matrix of $\widehat{S^{V'}}$, whereas
$\mathbf{B}$ is a generic matrix of $\widehat{\pi S}$.
Hence $\rk \mathbf{C}=\trk(S^{V'})=\dim_\F V'=p$. Besides $\rk \mathbf{C}+\rk \mathbf{B} \leq \rk \mathbf{M}$.

Consider the $\D$-linear subspace $L$ of $V$ spanned by the first $d$ vectors of $\bfB_V$, and the matrix
$\mathbf{M'}$ obtained by deleting the first $d$ rows of $\mathbf{M}$.
Then $\mathbf{M'}$ is a generic matrix of $\pi' S$ where $\pi' : V \twoheadrightarrow V/L$ is the standard projection.
By deleting the first $d$ rows we see that
$$\rk \mathbf{M'} \leq d(\dim_\D V'-1)+\rk \mathbf{B} \leq d(\dim_\D V'-1)+\rk \mathbf{M}-\rk \mathbf{C}=\rk \mathbf{M}-d.$$
Hence $\trk(\pi' S) \leq \trk(S)-d$. We conclude that $S$ is not weakly primitively intransitive.

The second point is deduced from the first one and from Proposition \ref{prop:strongtoweakprimitive}.
\end{proof}

\section{Nondegenerate alternators of deeply intransitive operator spaces}\label{section:bigproof}

Our aim here is to complete the proof of Theorem \ref{theo:deepintransitive}. We have already proved point (a).

Our proof mainly has two steps. The first and most difficult step consists, along with the proof of point (b) of Theorem \ref{theo:deepintransitive}, in proving
the following weak form of point (c):

\begin{prop}\label{prop:existenceofalternator}
Let $U$ and $V$ be finite-dimensional vector spaces over $\D$, with $n:=\dim V$.
Let $S$ be a deeply intransitive $\F$-linear subspace of $\Hom_\D(U,V)$ such that
$\dim S \geq \alpha(n,d)-n+\max(4-d,2)$.

Then $\Alt(S)$ is $1$-dimensional.
\end{prop}

Once this result is obtained, we will directly consider operator spaces that have a $1$-dimensional alternator space,
and point (c) of Theorem \ref{theo:deepintransitive} will be obtained along with Theorem
\ref{theo:deepintransitive2} thanks to the classification of multiplicative quadratic forms
on division algebras (Theorem \ref{theo:compositionalgebras}).

\subsection{Setup}\label{section:structureirrintro}

Throughout, we let $n \geq 1$ be an integer, and $U$ and $V$ be finite-dimensional $\D$-vector spaces, with $\dim_\D V=n$.
We assume that $|\F| \geq nd$. For simplicity, we set $\varepsilon=1$ if $d=1$, and $\varepsilon=0$ otherwise.
Finally, we take a deeply intransitive $\F$-linear subspace $S$ of $\Hom_\D(U,V)$
such that
$$\dim S \geq \alpha(n,d)-n+1+\varepsilon.$$
We will prove the following two points by induction on $n$:
\begin{itemize}
\item In any case $\trk(S)=nd-1$.
\item If $\dim S \geq \alpha(n,d)-n+2+\varepsilon$ then $S$ has a nonzero alternator.
\end{itemize}
Since $n>0$, we immediately retrieve from Corollary \ref{cor:sufficientconditiontargetreduced} that $S$ is target-reduced.

Throughout, we set
$$r:=\trk(S) \in \lcro 0,nd-1\rcro.$$
Since $S$ is target-reduced we have $S \neq \{0\}$ and hence $r>0$.
We fix a vector $x \in U \setminus \{0\}$ such that $\rk \widehat{x}=r$.
Then we apply the Tangent Space Lemma to recover that all the operators in
$$S':=\{u \in S : \; u(x)=0\}$$
map into $S x$. We denote by $H$ the socle of $S x$ and obtain that
all the operators in $S'$ map into $H$ (despite the notation, at this point $H$ is not necessarily a linear hyperplane of the $\D$-vector space $V$). We set
$$s:=\dim_\F H.$$
Note that $S'$ is included in $S^H$, which is deeply intransitive, and hence
$S'$ is also a deeply intransitive $\F$-linear subspace of $\Hom_\D(U,H)$.

Now it is time to represent the dual operator space $\widehat{S}$ in appropriate bases of $S$ and $V$.
We will choose such bases carefully with respect to the vector $x$.
For the basis of $V$:
\begin{itemize}
\item We start from a basis $(e_1,\dots,e_r)$ of $S x$ in which the first $s$ vectors span $H$
(as an $\F$-linear subspace), and we extend it arbitrarily into a basis $\bfE$ of $V$.
\item We lift the first vectors $e_1,\dots,e_r$ to obtain a family $(u_1,\dots,u_r)$ of operators in $S$,
so that $u_i(x)=e_i$ for all $i \in \lcro 1,r\rcro$. Then $(u_1,\dots,u_r)$ is linearly independent and
we can choose an arbitrary basis $(u_{r+1},\dots,u_{\dim S})$ of $S'$. Then our basis of $S$ is $\bfF:=(u_1,\dots,u_{r},u_{r+1},\dots,u_{\dim S})$.
\end{itemize}
We consider the matrix space $\calM$ that represents $\widehat{S}$ in the bases
$\bfF$ and $\bfE$. The matrix in $\calM$ that corresponds to the operator $\widehat{x}$ is the reduced matrix
$$J_r=\begin{bmatrix}
I_r & [0]_{r \times (\dim S -r)} \\
[0]_{(nd-r) \times r} & [0]_{(nd-r) \times (\dim S -r)}
\end{bmatrix}.$$
Now, we take $E:=U$ as our parameter space, $R$ and $\mathbb{L}$ as the associated ring and field, and we construct a generic matrix $\bfM$ of $\calM$ as follows:
for all $(i,j)\in \lcro 1,nd\rcro \times \lcro 1,\dim S\rcro$, we define
$\mathbf{m}_{i,j}$ as the linear form on $U$ that takes every vector $z \in U$ to the $(i,j)$-entry
of the matrix of $\widehat{z}$ in the bases $\bfF$ and $\bfE$.
Hence, the specialization $\bfM(z)$ is simply the matrix of $\widehat{z}$ in the bases $\bfF$ and $\bfE$.

As we already know that all the operators in $S'$ map into $H$, the matrix $\mathbf{M}$
actually takes the refined form
$$\mathbf{M}=\begin{bmatrix}
\mathbf{A} & \mathbf{C} \\
\mathbf{B} & [0]_{(nd-r) \times (\dim S-r)}
\end{bmatrix}$$
for some matrices $\mathbf{A} \in \Mat_{r}(R)$, $\mathbf{B} \in \Mat_{nd-r,r}(R)$
and $\mathbf{C} \in \Mat_{r,\dim S-r}(R)$ with $1$-homogeneous entries, and we can further refine
$$\mathbf{C}=\begin{bmatrix}
\mathbf{C_1} \\
[0]_{(r-s) \times (\dim S-r)}
\end{bmatrix} \qquad \text{with $\mathbf{C_1} \in \Mat_{s,\dim S-r}(R)$.}$$
Accordingly, we also refine
$$\mathbf{B}=\begin{bmatrix}
\mathbf{B_1} & \mathbf{B_2}
\end{bmatrix} \qquad \text{with $\mathbf{B_1} \in \Mat_{nd-r,s}(R)$ and $\mathbf{B_2} \in \Mat_{nd-r,r-s}(R)$.}$$
Finally, we write
$$\mathbf{A}=
\begin{bmatrix}
\mathbf{A_1} & \mathbf{A_3} \\
\mathbf{A_2} & \mathbf{A_4}
\end{bmatrix} \qquad \text{with $\mathbf{A_1} \in \Mat_{s}(R)$, $\mathbf{A_2} \in \Mat_{r-s,s}(R)$ and so on.}$$
Critically, $\mathbf{C_1}$ is a generic matrix for the dual operator space $\widehat{S'} \subseteq \Hom_\F(S',H)$ in the basis
$(u_{r+1},\dots,u_{\dim S})$ of $S'$
and the basis $(e_1,\dots,e_{s})$ of the $\F$-vector space $H$. Here, the parameter space for this generic matrix remains $U$.

By the very definition of the transitive rank, the maximal rank in $\calM$ is $r$. Since both matrices $J_r$ and $\mathbf{M}$
belong to $\Vect_\mathbb{L}(\calM)$ (see the remark following Corollary \ref{cor:genericrank}), so does every linear combination of them over $\mathbb{L}$.
Hence by Lemma \ref{prop:spanrank} every linear combination of $J_r$ and $\mathbf{M}$ over $\mathbb{L}$ has rank at most $r$.
Thanks to the infiniteness of $\mathbb{L}$ we can apply the Flanders-Atkinson Lemma (Lemma \ref{lemma:FlandersAtkinson}), yielding the identity
\begin{equation}\label{eq:FLA}
\forall k \geq 0, \; \mathbf{B} \mathbf{A}^k \mathbf{C}=0.
\end{equation}
The case $k=0$ yields
$$\mathbf{B_1} \mathbf{C_1}=0.$$

Before we proceed, it is time to complete the special case $n=1$ for the proof that $\trk(S)=nd-1$.
So, assume for a moment that $n=1$.
In that case $S'=\{0\}$ and hence $u \in S \mapsto u(x) \in S x$ is an $\F$-linear isomorphism.
It follows that $r=\dim S$. The assumption that $\dim S \geq \alpha(n,d)-n+1+\varepsilon$ requires here that $\dim S \geq nd-1$
and hence $r=nd-1$.

From now on, we assume that $n \geq 2$.
We will now draw some consequences of the assumption that $\dim S \geq \alpha(n,d)-n+1+\varepsilon$.

\begin{claim}\label{claim:S'targetreduced}
Assume that $\dim S \geq \alpha(n,d)-n+2+\varepsilon$ or $r<nd-1$.
Then the space $H$ has dimension $n-1$ over $\D$, and the space $S' \subseteq \Hom_\D(U,H)$ is target-reduced.
\end{claim}

\begin{proof}
Assume that $\dim_\D H<n-1$ or that $S'$ is not target-reduced.
Recall from Remark \ref{remark:Dtargetreduced} that this yields a $\D$-linear subspace $G$ of $V$ with dimension $n-2$ such that
$S' \subseteq \Hom_\D(U,G)$. Since $S'$ is deeply intransitive we deduce from point (a) of Theorem \ref{theo:deepintransitive}
that $\dim S' \leq \alpha(n-2,d)$. Hence $\dim S-r \leq \alpha(n-2,d)$.
Yet the starting assumption shows that $\dim S-r \geq \alpha(n-1,d)-n+2+\varepsilon$, and hence
$\alpha(n-2,d) \geq \alpha(n-1,d)-n+2+\varepsilon$, to the effect that $(n-1)d-1=\alpha(n-1,d)-\alpha(n-2,d) \leq n-2-\varepsilon$.
Hence $(n-1)(d-1) \leq -\varepsilon$, which contradicts the definition of $\varepsilon$ because $n \geq 2$.
\end{proof}

\begin{claim}\label{claim:B1targetreduced}
There is no row $L \in \Mat_{1,nd-r}(\F) \setminus \{0\}$ such that $L \mathbf{B_1}=0$. In particular $\mathbf{B_1} \neq 0$.
\end{claim}

\begin{proof}
Assume that the contrary holds. Then there is a linear hyperplane $\calH$ of the $\F$-vector space $V$ that includes $H$
and such that all the operators $u_1,\dots,u_{(n-1)d}$ map into $\calH$.
Denote by $H'$ the socle of $\calH$, and note that $H \subseteq H'$ and $H'$ is a $\D$-linear hyperplane of $V$.

The $\D$-linear mappings $u_1,\dots,u_{(n-1)d}$ map into $H'$, and so do all the elements of $S'$.
It follows that $\dim(S^{H'}) \geq \dim S'+(n-1)d$. Yet
$S^{H'}$ is deeply intransitive, so we can apply point (a) of Theorem \ref{theo:deepintransitive} once
more to find $\dim S'+(n-1)d \leq \alpha(n-1,d)$, leading to
$\dim S \leq \alpha(n-1,d)-(n-1)d+r=\alpha(n-2,d)+r-1$. A contradiction is then found like in the proof of Claim
\ref{claim:S'targetreduced}, using only $\dim S-r \geq \alpha(n-1,d)-n+1+\varepsilon$ here.
\end{proof}

\begin{claim}\label{claim:spanningrankmorethan1}
Assume that $\dim S \geq \alpha(n,d)-n+2+\varepsilon$ or $r<nd-1$.
Let $X \in \Mat_{1,nd-r}(\F) \setminus \{0\}$. Then $\sprk(X\mathbf{B_1})>1$.
\end{claim}

\begin{proof}
Assume that the contrary holds. By Claim \ref{claim:B1targetreduced}, we have $X \mathbf{B_1} \neq 0$,
and then $X\mathbf{B_1}=\mathbf{a} R_1$ for some row $R_1 \in \Mat_{1,(n-1)d}(\F) \setminus \{0\}$
and some $\mathbf{a} \in R \setminus \{0\}$.
The identity $\mathbf{B_1} \mathbf{C_1}=0$ then yields $\mathbf{a} (R_1 \mathbf{C_1})=0$, whence $R_1 \mathbf{C_1}=0$.

This contradicts the fact that $S'$ is target-reduced (see Claim \ref{claim:S'targetreduced}).
\end{proof}

\subsection{The transitive rank}\label{section:transitiverank}

\begin{claim}
One has $\trk(S)=nd-1$.
\end{claim}

\begin{proof}
We assume that $r\leq nd-2$ and seek to find a contradiction.

Denote by $r'$ the rank of $\mathbf{C_1}$: it is also the transitive rank of $S'$.
If $r'<(n-1)d-1$, then by induction $\dim S' < \alpha(n-1,d)-(n-1)+1+\varepsilon$
and hence
$$\dim S=r+\dim S' < nd-2+\alpha(n-1,d)-(n-1)+1+\varepsilon=\alpha(n,d)-n+1+\varepsilon,$$
which contradicts our starting assumption.

Hence $r'=(n-1)d-1$ (remember that $S'$ is intransitive).
Let us pick $\F$-linearly independent rows $L_1$ and $L_2$ in $\Mat_{1,nd-r}(\F)$.
Set $\mathbf{R_i}:=L_i \mathbf{B_1}$ for $i \in \{1,2\}$.
We have $\mathbf{R_i}\mathbf{C_1}=0$ with $\rk \mathbf{C_1}=(n-1)d-1$ and the number of rows of $\mathbf{C_1}$ is at most $(n-1)d$.
Hence $\mathbf{R_1}$ and $\mathbf{R_2}$ are linearly dependent over $\mathbb{L}$.
By Claim \ref{claim:spanningrankmorethan1} (remembering that we assume $r<nd-1$) we know that
$\sprk(\mathbf{R_1})>1$, and hence the Factorization Lemma yields
a scalar $\alpha \in \F$ such that $\mathbf{R_2}=\alpha \mathbf{R_1}$. Thus $(L_2-\alpha L_1) \mathbf{B_1}=0$
with $L_2-\alpha L_1 \in \Mat_{1,nd-r}(\F) \setminus \{0\}$, thereby contradicting Claim \ref{claim:B1targetreduced}.

We conclude that $r=nd-1$.
\end{proof}

\subsection{Deciphering the large spaces further}\label{section:constructalternator}

Now we proceed to prove the last point. So, now we assume that $\dim S \geq \alpha(n,d)-n+2+\varepsilon$,
and we prove that $S$ has a nonzero alternator. Note that the assumptions of Claims \ref{claim:S'targetreduced}
and \ref{claim:spanningrankmorethan1} are satisfied. In particular $S'$ is target-reduced.
Next,
$$\dim S'=\dim S-r \geq \alpha(n-1,d)-(n-1)+1+\varepsilon,$$
and hence by induction we recover $\trk(S')=(n-1)d-1$.
It then follows from Proposition \ref{prop:quasitransitive} that $\dim \Alt(S') \leq 1$.

Since $r=nd-1$, we see that $\mathbf{B_1}$ and $\mathbf{B_2}$ are simply row matrices,
and we know from Claim \ref{claim:spanningrankmorethan1} that
$\sprk(\mathbf{B_1})>1$.

Since $S'$ is target-reduced, the identity $\mathbf{B_1} \mathbf{C_1}=0$
yields that $\mathbf{B_1}$ represents a nonzero alternator $b : U \times H \rightarrow \F$ of $S'$ (seen as a subspace of $\Hom_\D(U,H)$).
We deduce that $\Alt(S')=\F b$, and hence the left radical
$$\calL:=\Lrad(b')$$
does not depend on the choice of $b'$ in $\Alt(S) \setminus \{0\}$
(in particular it is the left radical of $b$).

For all $y \in U$, the specialized matrix
$\mathbf{B_1}(y)$ represents the $\F$-linear form $b(y,-)$ in the basis $(e_1,\dots,e_{(n-1)d})$ of $H$.

Then for all $y \in U \setminus \calL$ we have $\mathbf{B_1}(y) \neq 0$, which is interpreted as follows in geometric terms:

\begin{claim}\label{claim:subtlequotient}
Let $y \in U \setminus \calL$. The mappings $u \in S$ such that $u(x) \in H$ do not all map $y$ into $S x$.
\end{claim}

We now move on to what is arguably the most subtle item in our proof:

\begin{claim}\label{claim:compoundrankd}
The compound matrix
$$\mathbf{N}:=\begin{bmatrix}
\mathbf{A_2} \\
\mathbf{B_1}
\end{bmatrix} \in \Mat_{d,(n-1)d}(R)$$
has rank $d$.
\end{claim}

\begin{proof}
We take an arbitrary vector $y \in U \setminus \calL$ (which exists!).
We shall prove that $\rk \mathbf{N}(y) \geq d$: since $(n-1)d \geq d$ this will be enough to obtain the conclusion.

Assume on the contrary that $\rk \mathbf{N}(y)<d$. Then there is a nonzero row
$L=\begin{bmatrix}
l_1 \cdots  l_d
\end{bmatrix} \in \Mat_{1,d}(\F)$ such that $L \mathbf{N}(y)=0$.
Denote by $\ell$ the $\F$-linear form on $V$ such that $\ell(e_1)=\dots=\ell(e_{(n-1)d})=0$ and
$\ell(e_{(n-1)d+i})=l_i$ for all $i \in \lcro 1,d\rcro$. This linear form is nonzero and its kernel includes $H$.
The assumption $L \mathbf{N}(y)=0$ then shows that all the operators $u_1,\dots,u_{(n-1)d}$ map $y$ into $\Ker \ell$,
and hence all the operators of $S$ that map $x$ into $H$ also map $y$ into $\Ker \ell$.

Next, we prove that $\Ker \ell=(S x)a^{-1}$ for some $a \in \D \setminus \{0\}$. Choose indeed an $\F$-linear form $\ell'$ on $V$
with kernel $S x$. The linear forms $z\in V \mapsto \ell'(za)$, with $a \in \D$, constitute a $d$-dimensional $\F$-linear subspace of $\Hom_\F(V,\F)$, and they
all vanish on $H$. By orthogonality theory, it follows that they are exactly the $\F$-linear forms on $V$ whose kernel include $H$.
In particular there exists $a \in \D$ such that $\forall z \in V, \; \ell(z)=\ell'(za)$, and obviously $a \neq 0$.
It follows that
$$\Ker \ell=(\Ker \ell')a^{-1}=(S x)a^{-1}=S (xa^{-1}).$$
Now the trick is to use the fact that $xa^{-1}$ satisfies exactly the same assumptions as the vector $x$ we started from.
Indeed $\rk (\widehat{xa^{-1}})=\dim_\F \bigl(S (xa^{-1})\bigr)=nd-1$, the operators in $S$ that vanish at $x$ are also the ones that vanish at $x a^{-1}$,
and $H$ is the $\D$-linear hyperplane of $V$ that is included in $S (xa^{-1})$. Hence Claim \ref{claim:subtlequotient} applies
when $x$ is replaced with $xa^{-1}$. Every $u \in S$ that maps $x$ into $H$ also maps $xa^{-1}$ into $H$ because $u$ is $\D$-linear,
and hence there is at least one such operator that does not map $y$ into $S (xa^{-1})$. This contradicts the first step of the present proof.
We conclude that $\rk \mathbf{N}(y)\geq d$, which completes the proof.
\end{proof}

\subsection{Constructing a nonzero alternator of $S$}

Now everything is in place for us to produce a nonzero alternator of $S$.

We start by applying identity \eqref{eq:FLA} once more. This time around, by taking $k=1$ we find
$\mathbf{B} \mathbf{A} \mathbf{C}=0$.
Writing
$$\mathbf{B} \mathbf{A}=\begin{bmatrix}
\mathbf{B'_1} & \mathbf{B'_2}
\end{bmatrix} \quad \text{with $\mathbf{B'_1} \in \Mat_{1,(n-1)d}(R)$ and $\mathbf{B'_2} \in \Mat_{1,d-1}(R)$,}$$
we deduce that $\mathbf{B'_1}$ is a left-annihilator of $\mathbf{C}_1$.
Note that all the entries of  $\mathbf{B'_1}$ are $2$-homogeneous.
Yet $\mathbf{B_1}$ is also a nonzero left-annihilator of $\mathbf{C}_1$.
Since $\rk \mathbf{C}_1=(n-1)d-1$, we deduce that $\mathbf{B'_1}$ is a scalar multiple of $\mathbf{B_1}$ over the field $\mathbb{L}$.
Thanks to $\sprk \mathbf{B_1}>1$ (see Claim \ref{claim:spanningrankmorethan1}) we deduce from the Factorization Lemma (Lemma \ref{lemma:genericcollinearity}) that
$$\mathbf{B'_1}=\mathbf{a} \mathbf{B_1} \quad \text{for some $1$-homogeneous $\mathbf{a} \in R$.}$$
Now, we consider the new row
$\mathbf{R}:=\mathbf{B}(\mathbf{A}-\mathbf{a} I_{nd-1})=\begin{bmatrix}
[0]_{1 \times (n-1)d} & \mathbf{R_2}
\end{bmatrix}$.
We also write
$$\mathbf{R} \mathbf{A}=\begin{bmatrix}
\mathbf{R'_1} & \mathbf{R'_2}
\end{bmatrix}$$
with $\mathbf{R'_1} \in \Mat_{1,(n-1)d}(R)$ and $\mathbf{R'_2} \in \Mat_{1,d-1}(R)$, both with $3$-homogeneous entries.
Using identity \eqref{eq:FLA} with $k=1$ and $k=2$, we now obtain
$$\mathbf{R} \mathbf{A} \mathbf{C}=0.$$
Applying the Factorization Lemma once more, we obtain
$$\mathbf{R'_1}=\mathbf{p} \mathbf{B_1} \quad \text{for some $2$-homogeneous $\mathbf{p} \in R$,}$$
which yields
$$\begin{bmatrix}
\mathbf{R_2} & -\mathbf{p}
\end{bmatrix}  \begin{bmatrix}
\mathbf{A_2} \\
\mathbf{B_1}
\end{bmatrix}=0.$$
Using Claim \ref{claim:compoundrankd}, we deduce that
$\begin{bmatrix}
\mathbf{R_2} & -\mathbf{p}
\end{bmatrix}=0$. In particular $\mathbf{R_2}=0$, and we conclude that
$$\mathbf{B}(\mathbf{A}-\mathbf{a} I_{nd-1})=0.$$
By combining this with $\mathbf{B} \mathbf{C}=0$, we conclude that
$\begin{bmatrix}
\mathbf{B} & -\mathbf{a}
\end{bmatrix}$ is a (nonzero) catcher of $\mathbf{M}$.

Since it was known from the start that $S$ is target-reduced,
the process explained in the end of Section \ref{section:genericmatrices} yields a nonzero alternator $\widetilde{b} : U \times V \rightarrow \F$ of $S$, associated
with the nonzero catcher $\begin{bmatrix}
\mathbf{B} & -\mathbf{a}
\end{bmatrix}$.

Finally, since $S$ has transitive rank $nd-1$ and is target-reduced, we know from Proposition \ref{prop:quasitransitive} that
$\dim \Alt(S) \leq 1$, and we conclude that $\dim \Alt(S)=1$.

This completes the proof of Proposition \ref{prop:existenceofalternator}.

\subsection{Operator spaces with $1$-dimensional alternator spaces}

We will now complete the proof of Theorem \ref{theo:deepintransitive}, by proving point (c).

We will prove the following general result, which avoids any dimensional assumption on $S$
and focuses on the condition regarding the structure of the alternator space.

\begin{prop}\label{prop:dimalt1toquadratic}
Let $S$ be an $\F$-linear subspace of $\Hom_\D(U,V)$ such that $\Alt(S)$ has dimension $1$, and
let $b \in \Alt(S) \setminus \{0\}$.
Then:
\begin{enumerate}[(a)]
\item $\D$ has quadratic type over $\F$.
\item If we take a pair $(\sigma,e)$ that is associated with $(\D,\F)$, then there is a unique $\sigma$-sesquilinear mapping
$B : U \times V \rightarrow \D$ such that $b(x,y)=e(B(x,y))$ for all $(x,y)\in U \times V$.
\end{enumerate}
\end{prop}

\begin{proof}
The key is to consider the ``action" of $\D$ on the set $\calB(U \times V,\F)$ of all bilinear forms on $U \times V$.
For such a form $c$ and for $a \in \D$, we set
$$c^a : (x,y)\in U \times V \mapsto c(xa,ya) \in \F.$$
Then $a \mapsto [c \mapsto c^a]$ defines an $\F$-quadratic mapping $P : \D \rightarrow \End_\F(\calB(U \times V,\F))$,
with $P(aa')=P(a) \circ P(a')$ for all $(a,a')\in \D^2$, i.e., $P$ is multiplicative.
Moreover, $P(a)$ leaves $\Alt(S)$ invariant for all $a \in \D$, as for all $c \in \Alt(S)$ we observe that
$$\forall u \in S, \; \forall x \in U, \; c(xa,u(x)a)=c(xa,u(xa))=0.$$
Since $\Alt(S)$ has dimension $1$, we deduce that $b$ is an eigenvector of $P(a)$ for all $a \in \D$, and by denoting by $q(a) \in \F$ the corresponding eigenvalue
we obtain an $\F$-quadratic form $q : \D \rightarrow \F$. Obviously $\rk b^a=\rk b$ for all $a \in \D^\times$
(because $\D$ is a division ring), and in particular $q$ is nonisotropic.
Finally, since $P$ is multiplicative we obtain that $q$ is multiplicative.

Therefore, we are in the position to use Theorem \ref{theo:compositionalgebras}: the algebra
$\D$ has quadratic type over $\F$. Let us now choose a pair $(\sigma,e)$ that is associated with $(\D,\F)$.
Then we also know that $q(a)=a\sigma(a)$ for all $a \in \D$.
Moreover $e$ satisfies the identity $\forall (x,y) \in \D^2, \; e(yx)=e(xy)$, which will ease things up from now on.
As seen in Section \ref{section:basicD}, since $e$ is nonzero we recover
a unique $\F$-bilinear and right-$\D$-linear form $B$ on $U \times V$ such that
$$\forall (x,y)\in U \times V, \; b(x,y)=e(B(x,y)).$$
Let $a \in \D$. Then
$$\forall (x,y)\in U \times V, \; b^a(x,y)=e(B(xa,ya))=e(B(xa,y)a)=e(a B(xa,y))$$
and we deduce that
$$\forall (x,y)\in U \times V, \; a B(xa,y)=q(a)B(x,y)=a \sigma(a) B(x,y).$$
If $a \neq 0$ we derive that
$$\forall (x,y)\in U \times V, \; B(xa,y)=\sigma(a) B(x,y),$$
which is of course obvious if $a=0$.
Hence $B$ is $\sigma$-sesquilinear, as claimed.
\end{proof}

Now we can complete the proof of point (c) of Theorem \ref{theo:deepintransitive}.
Let $S$ be a deeply intransitive $\F$-linear subspace of
$\Hom_\D(U,V)$ with $\dim S \geq \alpha(n,d)-n+\max(4-d,2)$ and $|\F| \geq nd$.
By Proposition \ref{prop:existenceofalternator} we have $\dim \Alt(S)=1$.

Choose $b \in \Alt(S) \setminus \{0\}$. We need to prove that $b$ is right-nondegenerate.
Set $R:=\Rrad(b)$, which by point (c) of Proposition \ref{prop:dimalt1toquadratic} is a $\D$-linear subspace of $V$.
Set $t:=\dim_\D R \in \lcro 0,n-1\rcro$, and denote by $\pi : V \twoheadrightarrow V/R$ the standard projection, which is $\D$-linear.
The bilinear form $b$ induces a right-nondegenerate $\F$-bilinear form $\overline{b}$ on $U \times (V/R)$.
The space $\pi S$ is an $\F$-linear subspace of $\Hom_\D(U,V/R)$, and it is obvious that
$\overline{b}$ is an alternator of it. Since $\overline{b}$ is right-nondegenerate,
we know from Proposition \ref{prop:fromalternatortointransitive} that $\pi S$ is deeply intransitive.
Besides, the kernel of $u \in S \mapsto \pi \circ u$ is precisely $S^R$, which is deeply intransitive. Hence, by combining the rank theorem with
point (a) of Theorem \ref{theo:deepintransitive}, we find
$$\dim (S)  = \dim (\pi S)+\dim(S^R) \leq \alpha(n-t,d)+\alpha(t,d) = \alpha(n,d)-dt(n-t).$$
It $t>0$ then $dt(n-t) \geq d(n-1)$, and we observe that $d(n-1)>n-\max(4-d,2)$ because obviously
$(d-1)(n-1)\geq 0 >1-\max(4-d,2)$. This runs counter to our starting assumption on the dimension of $S$.

Therefore $t=0$, to the effect that $b$ is right-nondegenerate.
This completes the proof of point (c) in Theorem \ref{theo:deepintransitive}.

Then, by combining this very point with Proposition \ref{prop:dimalt1toquadratic},
we derive Theorem \ref{theo:deepintransitive2}.

\subsection{The converse statement}\label{section:alternator}

It remains to check that, in the case of a division algebra of quadratic type, the spaces
$\calA_{b,\D}$ truly have the claimed properties (and most critically that their dimension is the bound $\alpha(n,d)$ we have found).
Interestingly, we can prove this without any assumption on the underlying field $\F$. This requires the next lemma, which for fields
with large cardinality is a straightforward consequence of Corollary \ref{cor:sufficientconditiontargetreduced}.

\begin{lemma}\label{prop:dimalt<=1}
Let $U$ and $V$ be finite-dimensional $\D$-vector spaces, with $\dim_\D U=\dim_\D V>0$, and set $n:=\dim_\D U$.
Let $S$ be an $\F$-linear subspace of $\Hom_\D(U,V)$ such that $\dim_\F(S x)=nd-1$ for all $x \in U \setminus \{0\}$.
Assume that $S$ is target-reduced.
Then all the nonzero alternators of $S$ are nondegenerate, and $\Alt(S)$ has dimension at most $1$.
\end{lemma}

\begin{proof}
Let $b \in \Alt(S)$ be nonzero and degenerate. It follows that $nd=\dim_\F V>1$.
Every element of $S$ maps the left radical $L$ of $b$ into its right radical $R$, and here
$\dim_\F L=\dim_\F R$ since $\dim_\F U=\dim_\F V$.

We choose $x \in L \setminus \{0\}$ and deduce that $\dim_\F R \geq \dim_\F S x=nd-1$. Hence $\rk(b)=1$.
Then each bilinear form $(x,y) \in U^2 \mapsto b(x,u(y))$, with $u \in S$, is alternating with rank at most $1$, and hence is zero.
Hence $\im u \subseteq R$ for all $u \in S$, which contradicts the assumption that $S$ is target-reduced.
Hence, every nonzero element of $\Alt(S)$ is nondegenerate.

Next, let $b$ and $b'$ belong to $\Alt(S) \setminus \{0\}$.
Choose $x \in U \setminus \{0\}$. By the first step, $b(x,-) \neq 0$. The nonzero linear forms $b'(x,-)$ and $b(x,-)$
vanish on the $\F$-linear hyperplane $S x$, and hence $b'(x,-)=\lambda b(x,-)$ for some $\lambda \in \F$.
Then $b'-\lambda b$ is a degenerate element in $\Alt(S)$, and by the first step $b'-\lambda b=0$.
We conclude that $\dim \Alt(S) \leq 1$.
\end{proof}

\begin{prop}\label{prop:conversequadratic}
Assume that $\D$ is of quadratic type over $\F$, and denote by $(\sigma,e)$ an associated pair.
Let $U$ and $V$ be finite-dimensional $\D$-vector spaces, with $n:=\dim_\D V>0$.
Let $B : U \times V \rightarrow \D$ be a right-nondegenerate $\sigma$-sesquilinear form, and
set $b : (x,y)\in U \times V \longmapsto e(B(x,y))$.
Then:
\begin{enumerate}[(a)]
\item $\dim \calA_{b,\D}=\alpha(n,d)$.
\item If $\dim_\D U=n$ then $\Alt(\calA_{b,\D})=\F b$ and $\dim (\calA_{b,\D} \, x)=nd-1$ for all $x \in U \setminus \{0\}$.
\item If $U=V$, then $\calA_{b,\D}$ has trivial spectrum if and only if $B$ is $e$-nonisotropic, in which case it is irreducible.
\end{enumerate}
\end{prop}

Note that point (a) is enough to yield Proposition \ref{prop:deepintransitiveconverse}. Indeed, in the conclusion, the deep intransitivity of
$\calA_{b,\D}$ is a straightforward consequence of point (b) in Proposition \ref{prop:fromalternatortointransitive}
(remember from Lemma \ref{lemma:relationradicals} that the right-nondegeneracy of $b$ is equivalent to the one of $B$).

\begin{proof}
Set $S:=\calA_{b,\D}$.
To start with, to prove the first point we reduce the situation to the one where $B$ is nondegenerate.
To see this, we consider the left radical $L$ of $B$, which equals the one of $b$, and the induced form $\overline{B} : (U/L) \times V \rightarrow \mathbb{D}$,
and the $\F$-bilinear form $\overline{b} : (x,y)\in (U/L) \times V \mapsto e(\overline{B}(x,y))$.
All the operators in $S$ map $L$ into the right radical of $b$, which is zero, so they all vanish on $L$.
Hence every $u \in S$ induces a $\D$-linear mapping $\overline{u} : U/L \rightarrow V$,
and we have $\overline{b}(\overline{x},\overline{u}(\overline{x}))=0$ for all $x \in U/L$. We denote by $\overline{S}$ the space of all mappings
$\overline{u}$ with $u \in S$.

Conversely, for all $u \in \calA_{\overline{b},\D}$ it is clear that we have $u \circ \pi \in S$ for the standard projection $\pi : U \twoheadrightarrow U/L$.
Therefore $\overline{S}=\calA_{\overline{b},\D}$. Since $u \in S \mapsto \overline{u} \in \overline{S}$
is clearly injective, we deduce that $\dim \calA_{b,\D}=\dim \calA_{\overline{b},\D}$.

Hence, in the remainder of the proof we need only consider the situation where $\dim_\D U=\dim_\D V=n$.
We shall use a matrix interpretation.
Let $x \in U \setminus \{0\}$, and extend it into a $\D$-basis $\bfB_U$ of $U$ with $x$ as first vector,
and consider a $\D$-basis $\bfB_V$ of $V$.

Denote by $P$ the Gram matrix of $B$ in the bases $U$ and $V$, and set $\calM=S\calH_n(\D)$.
Then the elements of $S$ are represented by the elements of $P^{-1} \calM$ in the bases $\bfB_U$ and $\bfB_V$, and in particular
$\dim \calA_{b,\D}=\dim \calM$.
Then, it is easily checked in each case that
$$\dim \calM=n(d-1)+d \dbinom{n}{2}=\alpha(n,d);$$
indeed, each diagonal entry varies independently in $\Ker e$, the strictly upper-diagonal
entries vary independently in $\D$ and the strictly lower-diagonal entries are the opposites of the conjugates of the strictly upper-diagonal ones
under $\sigma$.

Next, denote by $X$ the first vector of the standard basis of $\D^n$. Then we observe that $\calM X=(\Ker e) \times \D^{n-1}$
and hence $\dim_\F(S x)=nd-1$.
Moreover if $n \geq 2$ and we take $X'$ as the second vector of the standard basis we find $\calM X'=\D \times (\Ker e) \times \D^{n-2}$,
and it follows that $S x+S x'=V$ for some $x' \in U$.
If $d>1$ or $n>1$ we have collected enough information to see that $S$ is target-reduced, and then we conclude from
Lemma \ref{prop:dimalt<=1} that $\Alt(S)=\F b$.
If $d=1$ and $n=1$ then $S=\{0\}$ and $\Alt(S)$ has dimension $1$ (it is the space of all $\F$-bilinear forms on $U \times V$).

Assume finally that $U=V$. By definition, $B$ is $e$-nonisotropic if and only if $b$ is nonisotropic.
We already know that $\calA_{b,\D}$ has trivial spectrum if $b$ is nonisotropic.

Assume for the remainder of the proof that $\calA_{b,\D}$ has trivial spectrum.
\begin{itemize}
\item Let $x \in V \setminus \{0\}$. Then as $\dim_\F(S x)=nd-1$ and $S x$ is included in the right-orthogonal of $x$ under $b$, it equals this right-orthogonal,
which yields $b(x,x) \neq 0$ since $x \not\in S x$.
Hence $b$ is nonisotropic.
\item Let us finally prove that $S$ is irreducible. Let $x \in V \setminus \{0\}$. Combining
$\dim_\F(S x)=nd-1$ with $x \not\in S x$, we find $\F x\oplus S x=V$. Therefore, $S$ is irreducible.
\end{itemize}
\end{proof}

\section{Optimal trivial spectrum subspaces}\label{section:trivialspectrum}

In this section, we will quickly derive the main results on irreducible optimal trivial spectrum subspaces from their
analogues on deeply intransitive operator spaces.

\subsection{The Invariant Subspace Lemma}

We start with a simple lemma, which will allow us to deal with the case of non-primitively intransitive spaces.

\begin{lemma}[Invariant Subspace Lemma]\label{lemma:stab}
Let $S$ be a trivial spectrum subspace of $\End_\D(V)$.
Assume that there exists a $\D$-linear subspace $W$ of $V$ such that $S$ contains every operator $u \in \End_\D(V)$ that satisfies
$\im u \subseteq W \subseteq \Ker u$. Then $W$ is invariant under $S$.
\end{lemma}

\begin{proof}
Let $u \in S$. Assume that there is a (nonzero) vector $x \in W$ such that $u(x)\not\in W$.
The vector $u(x)-x$ is not in $W$, and hence it belongs to a direct summand $W'$ of $W$ in $V$ (i.e., a $\D$-linear subspace of $V$ such that $V=W \oplus W'$).
Denote by $\pi$ the projection of $V$ onto $W$ along $W'$, and by $v \in \End_\D(V)$ the linear mapping that
vanishes on $W$ and satisfies $v(y)=\pi(u(y))$ for all $y \in W'$.
Then $v \in S$ by our assumptions, and hence $u-v \in S$. Note that $u-v$ leaves $W'$ invariant.
Moreover $(u-v)(x) = u(x) \equiv x$ mod $W'$.
The endomorphism of $V/W'$ induced by $u-v$ then has a nonzero fixed vector, which contradicts the assumption that it should have trivial spectrum.
Therefore $W$ is invariant under $S$.
\end{proof}

\subsection{The partial structure theorem for optimal spaces}

Now, we shall complete the proof of Theorem \ref{theo:partialclassificationoptimal}.
The case $n=1$ is obvious.
So, we let $n \geq 2$ be an integer, $V$ be an $n$-dimensional $\D$-vector space, and
$S$ be an irreducible trivial spectrum subspace of $\End_\D(V)$ with dimension $\alpha(n,d)$.
We also assume that $|\F| \geq nd$.

We immediately note that $S$ is intransitive.

\vskip 3mm
\noindent \textbf{Case 1. The space $S$ is primitively intransitive.} \\
Then $S$ is deeply intransitive, by Proposition \ref{prop:primitiveintransitivitytodeep}.
\begin{itemize}
\item Assume first that $(n,d) \neq (2,1)$.
Then $n \geq \max(4-d,2)$, and hence $\dim S \geq \alpha(n,d)-n+\max(4-d,2)$.
By point (c) of Theorem \ref{theo:deepintransitive}, we find that $S$ has a nondegenerate alternator $b$.

\item Assume now that $(n,d)=(2,1)$. Then we simply note that $\dim S=1$, so we can write $S=\F u_0$ for some
operator $u_0 \in \End_\D(V)$;
then $u_0$ is invertible because of the irreducibility assumption, and hence by picking an arbitrary
symplectic form $s$ on the $\F$-vector space $V$
(which exists because $n=2$ and $d=1$), we see that $b : (x,y) \mapsto s(x,u_0^{-1}(y))$ is a nondegenerate alternator of $S$.
\end{itemize}

In any case, we have found a nondegenerate alternator $b$ of $S$, to the effect that $S \subseteq \calA_{b,\D}$.
Yet $\calA_{b,\D}$ is deeply intransitive (by point (b) of Proposition \ref{prop:fromalternatortointransitive}), and hence
$\dim \calA_{b,\D} \leq \alpha(n,d)$ by point (a) of Theorem \ref{theo:deepintransitive}. Therefore
$S=\calA_{b,\D}$.

\vskip 3mm
\noindent \textbf{Case 2. The space $S$ is not primitively intransitive.} \\
Then we take a maximal $\D$-linear subspace $V'$ of $V$ such that, for the standard projection
$\pi : V \twoheadrightarrow V/V'$, the projected space $\pi S$ is intransitive. Note that $\{0\} \subset V' \subset V$.
By the maximality of $V'$, the space $\pi S$ is primitively intransitive, and hence deeply intransitive
(see Proposition \ref{prop:primitiveintransitivitytodeep}).
Set then $t:=\dim_\D V'$.
Applying point (a) of Theorem \ref{theo:deepintransitive}, we find $\dim(\pi S) \leq \alpha(n-t,d)$.
Now, we look at the space $S^{V'}$. First of all, the rank theorem yields
$$\dim(S^{V'})=\dim (S)-\dim (\pi S)\geq \alpha(n,d)-\alpha(n-t,d)=d t(n-t)+\alpha(t,d).$$
Assigning to every $f \in S^{V'}$ its restriction $f_{V'} \in \End_\D(V')$, we create an $\F$-linear mapping
$\Phi$ whose range is a trivial spectrum subspace of $\End_\D(V')$ and whose kernel is the subspace
$\calD$ of all $f \in S^{V'}$ that vanish on $V'$.
By Theorem \ref{theo:majodimsimple} we have $\rk \Phi \leq \alpha(t,d)$, whence the rank theorem yields
$$\dim \calD \geq \dim (S^{V'})-\alpha(t,d) \geq d t(n-t).$$
Yet $\calD$ is included in the space of all $\D$-endomorphisms of $V$ that vanish on $V'$ and map into $V'$, a space
which has dimension $dt(n-t)$ over $\F$. Hence $\calD$ equals that space, to the effect that
all such endomorphisms belong to $S$.

Then the Invariant Subspace Lemma yields that $V'$ is invariant under $S$. However this contradicts the assumption that $S$
is irreducible, and we conclude that Case 2 is actually impossible.

Hence, Theorem \ref{theo:partialclassificationoptimal} is now proved.

\subsection{The full structure theorem for optimal spaces}

Now, we turn to Theorem \ref{theo:fullclassificationoptimal}.

Let $V$ be a (right) $\D$-vector space with finite dimension $n>0$. Assume that $|\F|\geq nd$,
and let $b$ be a nonisotropic $\F$-bilinear form on $V$
such that $S:=\calA_{b,\D}$ is an optimal irreducible trivial spectrum subspace of $\End_\D(V)$.

We shall prove that $\dim \Alt(S)=1$ in any case.
\begin{itemize}
\item Assume first that $n \geq \max(4-d,2)$.
Then $\dim S=\alpha(n,d) \geq \alpha(n,d)-n+\max(4-d,2)$ and the claimed statement is deduced from
point (c) of  Theorem \ref{theo:deepintransitive}.
\item Assume now that $n=2$ and $d=1$.
Then $\dim S=1$, and we take $u_0 \in S \setminus \{0\}$. Since $S$ is irreducible, $u_0$ must be invertible,
and we deduce that $\dim S x=1=nd-1$ for all $x \in V \setminus \{0\}$. Moreover $S$ is target-reduced because $u_0$ is surjective.
Hence Lemma \ref{prop:dimalt<=1} shows that $\dim \Alt(S) \leq 1$, and we conclude that $\dim \Alt(S)=1$ because $b \in \Alt(S)$.
\item If $n=1$ and $d>1$ then $\dim_\F (S x)=\dim S=nd-1$ for all $x \in V \setminus \{0\}$, and obviously $S$ is target-reduced,
so once more we find by Proposition \ref{prop:dimalt<=1} that $\dim \Alt(S)=1$.
\item If $n=1$ and $d=1$ then $S=\{0\}$ and $\Alt(S)$ is the space of all $\F$-bilinear forms on $V$, which has dimension $1$.
\end{itemize}
Then, we apply Proposition \ref{prop:dimalt1toquadratic} and find that $\D$ has quadratic type over $\F$ and that,
for an associated pair $(\sigma,e)$, there exists a $\sigma$-sesquilinear form $B : V^2 \rightarrow \D$
such that $b : (x,y) \mapsto e(B(x,y))$. Besides, it has already been observed in Section \ref{section:basicD} that $B$ is uniquely determined by $b$.

Finally, the converse statement in Theorem \ref{theo:fullclassificationoptimal} has already been proved through Proposition \ref{prop:conversequadratic}.

\subsection{The quadratic type case}

We finish this section by proving the results that are specific to the division algebras of quadratic type: Theorems \ref{theo:partialclassificationoptimalwith1}
and \ref{theo:classuptosimilarity}.

So, assume that $\D$ has quadratic type over $\F$, and denote by $(\sigma,e)$ an associated pair.
Let $V$ be an $n$-dimensional vector space over $\D$, with $n \geq 1$.
Let $S$ be an irreducible optimal trivial spectrum subspace of $\End_\D(V)$ such that $|\F| \geq nd$.

If $n \geq 2$ then we apply Theorems \ref{theo:partialclassificationoptimal} and \ref{theo:fullclassificationoptimal} to find a $\sigma$-sesquilinear nondegenerate form
$B : V^2\rightarrow \D$ such that $S=\calA_{b,\D}$ for $b : (x,y) \mapsto e(B(x,y))$.

Assume now that $n=1$. In that case we need only consider the special case where $V=\D$.
By the isomorphism $a \in \D \mapsto [x \mapsto ax] \in \End_\D(\D)$ of $\F$-vector spaces, we associate to $S$ an $\F$-linear hyperplane $H$ of $\D$.
Thus $H$ is the kernel of $x \mapsto e(\alpha x)$ for some $\alpha \in \D^\times$. Then, for all $a \in H$, we have
$$\forall x \in \D, \; e(x^\sigma \alpha (ax))=e(\alpha a \underbrace{xx^\sigma}_{\in \F})=(xx^\sigma)\, e(\alpha a)=0$$
This shows that the nondegenerate bilinear form $b : (x,y) \in \D^2 \mapsto e(x^\sigma \alpha y)$
is an alternator of $S$. Hence $S \subseteq \calA_{b,\D}$. Yet $\dim \calA_{b,\D} \leq d-1$ because $\calA_{b,\D}$ is deeply intransitive,
and we conclude that $S=\calA_{b,\D}$. Note here that $b : (x,y) \mapsto e(B(x,y))$ for the $\sigma$-sesquilinear form
$B : (x,y) \mapsto x^\sigma \alpha y$. By point (c) of Proposition \ref{prop:conversequadratic}, we see that $b$ is nonisotropic.

From here, we deduce from point (b) of Proposition \ref{prop:conversequadratic} that
$\Alt(S)$ has dimension $1$: it follows that for every nondegenerate bilinear form $b'$ on $V$ such that $S=\calA_{b',\D}$,
we have $b' \in \Alt(S)$ and hence $b' \in \F b$.
This completes the proof of Theorem \ref{theo:partialclassificationoptimalwith1}.

Let us finally derive Theorem \ref{theo:classuptosimilarity}.
So, let $B$ and $B'$ be nondegenerate $\sigma$-sesquilinear forms on $V$. Let $\varphi \in \GL_\D(V)$ and set
$b : (x,y) \in V^2 \mapsto e(B(x,y))$ and $b' : (x,y) \in V^2 \mapsto e(B'(x,y))$.

Then, for all $u \in \End_\D(V)$,
\begin{align*}
\varphi \circ u \circ \varphi^{-1} \in \calA_b & \Leftrightarrow \forall x \in V, \; e(B(x,\varphi(u(\varphi^{-1}(x)))=0 \\
& \Leftrightarrow \forall y \in V, \; e(B(\varphi(y),\varphi(u(y)))=0 \\
& \Leftrightarrow u \in \calA_{b''}
\end{align*}
where $b'' : (x,y) \mapsto e(B''(x,y))$ for the $\sigma$-sesquilinear form $B'' : (x,y) \mapsto B(\varphi(x),\varphi(y))$.
Hence $\calA_{b'',\D}$ is the conjugate of $\calA_{b,\D}$ under $\varphi^{-1}$, and more generally
$\calA_{\alpha b'',\D}=\calA_{b'',\D}$ is similar to $\calA_{b,\D}$ for all $\alpha \in \F^\times$.

Conversely, assume that $\calA_{b,\D}$ is similar to $\calA_{b' ,\D}$ under an element of $\GL_\D(V)$.
Then by the above $\calA_{b',\D}=\calA_{b'',\D}$ for some $b''$ of the form $(x,y) \mapsto e(B''(x,y))$, where $B''$ is equivalent to $B$.
The uniqueness statement in Theorem \ref{theo:partialclassificationoptimalwith1} yields $b'=\alpha b''$ for some $\alpha \in \F^\times$,
and hence $B'=\alpha B''$. Hence $B'$ is equivalent to $\alpha B$ for some $\alpha \in \F^\times$.
This completes the proof of Theorem \ref{theo:classuptosimilarity}.

\section{Optimal affine spaces of units}\label{section:affinenonsingular}

In this final section, we solve our initial problem of unital affine subspaces of a simple algebra.

\subsection{The greatest possible dimension}

To start with, we have a straightforward application of Theorem \ref{theo:dimMax}:

\begin{theo}
Let $\D$ be a division ring with finite dimension $d$ over a central subfield $\F$, and let $n \geq 1$ be such that 
$|\F| \geq nd$. Then the greatest possible dimension for a unital $\F$-affine subspace of $\Mat_n(\D)$ is $n(d-1)+d \binom{n}{2}$.
\end{theo}

\subsection{The optimal spaces}

For the structure of optimal spaces, we will split the discussion, whether $\D$ has quadratic type over $\F$ or not.

\begin{theo}
Let $n \geq 1$. Assume that $\D$ is not of quadratic type over $\F$, and that $|\F| \geq nd$. Then, up to equivalence, there is exactly one
optimal unital $\F$-affine subspace of $\Mat_n(\D)$.
\end{theo}

One such space is the $n$-fold joint $(1_\D+H) \vee \cdots \vee (1_\D+H)$, where $H$ is an arbitrary $\F$-linear hyperplane of $\D$ that does not contain $1_\D$.

\begin{proof}
First of all it is clear that the $n$-fold joint $(1_\D+H) \vee \cdots \vee (1_\D+H)$ is an $\F$-affine subspace of invertible matrices of $\Mat_n(\D)$,
with dimension $n(d-1)+d\dbinom{n}{2}$.

Conversely, let $\calM$ be an optimal $\F$-affine subspace of invertible matrices of $\Mat_n(\D)$.
Replacing $\calM$ with an equivalent space if necessary, we can assume that $I_n \in \calM$.
Then the translation vector space $\overrightarrow{\calM}$ of $\calM$ is a trivial spectrum subspace of $\Mat_n(\D)$,
and an optimal one. By Theorem \ref{theo:classifnonquadratic}, it is similar to $H_1 \vee \cdots \vee H_n$ for some list $(H_1,\dots,H_n)$ of
$\F$-linear hyperplanes of $\D$, none of which contains $1_\D$.
Consequently $\calM \sim I_n+(H_1 \vee \cdots \vee H_n)=(1_\D+H_1) \vee \cdots \vee (1_\D+H_n)$.
To conclude, we note that any two $\F$-affine hyperplanes of $\D$ that do not go through $0_\D$ are equivalent.
Indeed, let $\calH$ and $\calH'$ be such hyperplanes. They have respective equations of the form $e(x)=1_\F$ and $f(x)=1_\F$, where
$e$ and $f$ are nonzero $\F$-linear forms on $\D$.
Then, as seen in Section \ref{section:duality}, there exists $a \in \D \setminus \{0\}$ such that $f : x \mapsto e(ax)$.
This yields $\calH'=a^{-1} \calH$, to the effect that $\calH \sim \calH'$.

Finally, $n$-fold joints are clearly compatible with equivalence, to the effect that for any list $\calH_1,\dots,\calH_n$ of
$\F$-affine hyperplanes of $\D$ that do not contain $0_\D$, we have
$\calM \sim \calH_1 \vee \cdots \vee \calH_n$. The claimed result ensues.
\end{proof}

The situation is both more complicated and more interesting when $\D$ has quadratic type over $\F$.

\begin{theo}[Classification of optimal $\F$-affine subspaces of invertible matrices]\label{theo:classaffine}
Let $n \geq 1$. Assume that $\D$ has quadratic type over $\F$, with $|\F| \geq nd$. Let $(\sigma,e)$ be an associated pair.

Let $\calM$ be an optimal unital $\F$-affine subspace of $\Mat_n(\D)$.
Then there exists a partition $n=n_1+\cdots+n_p$ and a list $(P_1,\dots,P_p)$ of $e$-nonisotropic matrices, with $P_i \in \GL_{n_i}(\D)$, such that
$$\calM \simeq (P_1+\calS\calH_{n_1}(\D) \vee \cdots \vee (P_p+\calS\calH_{n_p}(\D)).$$

Conversely, let $n=n_1+\cdots+n_p$ be a partition, and $(P_1,\dots,P_p)$ be a list of $e$-nonisotropic matrices, with $P_i \in \GL_{n_i}(\D)$.
Then $(P_1+\calS\calH_{n_1}(\D)) \vee \cdots \vee (P_p + \calS\calH_{n_p}(\D))$ is an optimal unital $\F$-affine subspace of $\Mat_n(\D)$.

Finally, let $m=m_1+\cdots+m_q$ be a partition, and $(P'_1,\dots,P'_q)$ be a list of $e$-nonisotropic matrices, with $P'_i \in \GL_{m_i}(\D)$.
Then $(P_1+\calS\calH_{n_1}(\D)) \vee \cdots \vee (P_p +\calS\calH_{n_p}(\D))$ is equivalent to
$(P'_1+\calS\calH_{m_1}(\D)) \vee \cdots \vee (P'_q + \calS\calH_{m_q}(\D))$ if and only if
$(n_1,\dots,n_p)=(m_1,\dots,m_q)$ and, for all $k \in \lcro 1,p\rcro$,
there exists an invertible matrix $Q_k \in \GL_{n_k}(\D)$
and a scalar $\alpha \in \F^\times$ such that
$\forall X\in \D^{n_k}, \;  e(X^\star P'_k X)=\alpha\,e((Q_kX)^\star P_k (Q_kX))$.
\end{theo}

The proof involves the following lemma:

\begin{lemma}\label{lemma:supertransitivity}
Let $S$ be an irreducible optimal trivial spectrum subspace of $\End_\D(V)$, where $V$ is an $n$-dimensional vector space over $\D$,
with $|\F| \geq nd$. Then $\F x+S x=V$ for all $x \in V \setminus \{0\}$.
\end{lemma}

\begin{proof}
We combine Theorems \ref{theo:fullclassificationoptimal} and \ref{theo:partialclassificationoptimalwith1} to find that
$S=\calA_{b,\D}$ for some nondegenerate $\F$-bilinear form $b$ of type $(x,y) \mapsto e(B(x,y))$, where $B$ is a nondegenerate $\sigma$-sesquilinear form on $V$.
Let $x \in V \setminus \{0\}$.
By point (b) of Proposition \ref{prop:conversequadratic} we find $\dim_\F (S x)=nd-1$. The conclusion is obtained by noting that $x \not\in S x$ because $S$ has trivial spectrum.
\end{proof}

\begin{proof}[Proof of Theorem \ref{theo:classaffine}]
To simplify the writing of the proof, we set $\calM_k:=\calS\calH_k(\D)$ for every integer $k \geq 1$.

Again, we lose no generality in assuming that $I_n \in \calM$. The translation vector space $\overrightarrow{\calM}$ of $\calM$
is then an optimal trivial spectrum subspace of $\Mat_n(\F)$.
This yields a partition $n=n_1+\cdots+n_p$ and a list $(P_1,\dots,P_p)$ of $e$-nonisotropic matrices, with $P_i \in \GL_{n_i}(\D)$, such that
$$\overrightarrow{\calM} \simeq (P_1^{-1}\calM_{n_1}) \vee \cdots \vee (P_p^{-1}\calM_{n_p}).$$
Hence
$$\calM \simeq I_n+\bigl((P_1^{-1}\calM_{n_1}) \vee \cdots \vee (P_p^{-1}\calM_{n_p})\bigr)
=(I_{n_1}+P_1^{-1}\calM_{n_1})\vee \cdots \vee (I_{n_p}+P_p^{-1}\calM_{n_p}),$$
leading to
$$\calM \sim (P_1+\calM_{n_1})\vee \cdots \vee (P_p+\calM_{n_p}).$$

The second statement is obvious because, for every $k \in \lcro 1,p\rcro$ and every $e$-nonisotropic $P \in \GL_k(\D)$
we have $P+\calM_k \sim I_k+P^{-1} \calM_k$, and $P^{-1} \calM_k$ has trivial spectrum and dimension $\alpha(k,d)$ over $\F$;
note also, for the remainder of the proof, that $\Vect_\F(P+\calM_k)X=\D^k$ for all $X \in \D^k \setminus \{0\}$, as a consequence of
Lemma \ref{lemma:supertransitivity}.

It remains to prove the uniqueness statement. It is better to think in geometrical terms here. So, let $V$ be an $n$-dimensional vector space over
$\D$, and let $S$ be an optimal $\F$-affine subspace of automorphisms of $V$.
Assume that we have bases $(e_1,\dots,e_n)$ and $(f_1,\dots,f_n)$ of $V$ such that the matrices that represent the elements of $S$ form the space
$(P_1+\calM_{n_1})\vee \cdots \vee (P_p+\calM_{n_p})$. We consider the integers $N_i:=\sum_{k=1}^i n_k$ and
the spaces $U_i:=\Vect(e_1,\dots,e_{N_i})$ and $V_i:=\Vect(f_1,\dots,f_{N_i})$ for $i \in \lcro 0,p\rcro$.
We shall prove that any other choice of bases will yield the same spaces.
To see this, we look at the dimension $\dim_\F(S x)$ when $x$ ranges over $V$.
Using the previous remark, we see that for all $i \in \lcro 1,p \rcro$ and all $x \in U_i \setminus U_{i-1}$,
we have $\Vect_\F(S x)=V_i$.
Hence the spaces $V_0,\dots,V_p$ are uniquely determined by $S$ only, as the possible $S x$ spaces when $x$ ranges over $V$;
and then $U_i=\{x \in V : \; S x \subseteq V_i\}$ is also uniquely determined by $S$ for all $i \in \lcro 1,p\rcro$.

As a consequence, if we have invertible matrices $R,S$ in $\GL_n(\D)$ such that
$$(P_1+\calM_{n_1}) \vee \cdots \vee (P_p + \calM_{n_p})=R \left((P'_1+\calM_{n'_1}) \vee \cdots \vee (P'_q + \calM_{n'_q})\right) S$$
then $(n_1,\dots,n_p)=(n'_1,\dots,n'_q)$ and both matrices $R$ and $S$ are block upper-triangular with respect to the partition
$(n_1,\dots,n_p)$. Their respective diagonal blocks are invertible, and it follows that
$$\forall i \in \lcro 1,p\rcro, \; P_i+\calM_{n_i} \sim P'_i+\calM_{n_i}.$$
Conversely, it is obvious that if $(n_1,\dots,n_p)=(n'_1,\dots,n'_q)$ and $\forall i \in \lcro 1,p\rcro, \; P_i+\calM_{n_i} \sim P'_i+\calM_{n_i}$
then $(P_1+\calM_{n_1}) \vee \cdots \vee (P_p + \calM_{n_p}) \sim (P'_1+\calM_{n'_1}) \vee \cdots \vee (P'_q + \calM_{n'_q})$.

In order to conclude, we go back to an arbitrary integer $k \in \lcro 1,n\rcro$, we take two $e$-nonisotropic matrices
$P$ and $Q$ of $\Mat_k(\D)$, and we characterize the equivalence $P+\calM_k \sim Q+\calM_k$.
First of all, note that $R^\star \calM_k R=\calM_k$ for all $R \in \GL_k(\D)$ (which easily follows from Proposition \ref{prop:skewHermitian}).
Now, say that we have $R,S$ in $\GL_k(\D)$ such that $P+\calM_k=R(Q+\calM_k) S$.
Then $P+\calM_k=RQS+R\calM_k S$. Comparing the translation vector spaces yields $R\calM_kS=\calM_k$
i.e., $R(S^\star)^{-1}(S^\star \calM_k S)=\calM_k$, leading to $R(S^\star)^{-1} \calM_k=\calM_k$.
We observe that $(X,Y) \mapsto e(X^\star S^\star R^{-1} Y)$ is a nondegenerate alternator of $R(S^\star)^{-1} \calM_k$,
that $(X,Y) \mapsto e(X^\star Y)$ is a nonzero alternator of $\calM_k$, and finally that $\dim_\F (\calM_k X)=kd-1$ for all $X \in \D^k \setminus \{0\}$.
Hence by Proposition \ref{prop:dimalt<=1} those alternators are $\F$-linearly dependent, yielding $S^\star R^{-1}=\alpha I_k$ for some $\alpha \in \F^\times$.
Hence $P+\calM_k=P'+\calM_k$ for $P':=\alpha^{-1} S^\star Q S$.
Finally, $P=P'+M$ for some $M \in \calM_k$, which yields $\forall X \in \D^n, \; e(X^\star PX)=e(X^\star P'X)=\alpha^{-1}\, e((SX)^\star Q (SX))$.

Conversely. assume that $\forall X \in \D^k, \; e(X^\star PX)=\alpha\,e((SX)^\star Q (SX))$ for some $\alpha \in \F^\times$
and some $S \in \GL_k(\D)$. Then $P-\alpha S^\star QS$ belongs to $\calM_k$, whence
$$P+\calM_k=\alpha S^\star QS+\calM_k \sim Q+(S^{\star})^{-1} \calM_k S^{-1} =Q +\calM_k.$$
This completes our proof.
\end{proof}

\subsection{Making sense of the classification of optimal spaces}

It remains to understand the nature of the mappings $X \in \D^n \mapsto e(X^\star PX)$ when $n \geq 1$ and $P$ ranges over the $e$-nonisotropic matrices of $\GL_n(\D)$.
In the special case $d=1$, they are simply the nonisotropic quadratic forms on $\F^n$.
Hence, as seen in \cite{dSPlargeaffinenonsingular}, the optimal unital $\F$-affine subspaces of
$\Mat_n(\F)$ are classified, up to equivalence, by lists of congruence classes of nonisotropic quadratic forms. In particular, if $\F=\R$
then there is only one such congruence class of $k$-dimensional nonisotropic quadratic forms for each $k \geq 1$, and we obtain that every
optimal affine subspace of invertible matrices of $\Mat_n(\R)$ is equivalent to
$(I_{n_1}+\Mata_{n_1}(\R)) \vee \cdots \vee (I_{n_p}+\Mata_{n_p}(\R))$ for a unique partition $n=n_1+\cdots+n_p$,
and conversely every space of this form is an optimal affine subspace of invertible matrices of $\Mat_n(\R)$.

The situation is not as clear if $d>1$. But say that $\car(\F) \neq 2$, and keep the assumption that $\D$ is of quadratic type over $\F$.
Then $d \in \{2,4\}$. Let $P \in \Mat_n(\F)$. Then there is a unique splitting $P=Q+R$ where $Q$ is Hermitian (that is $Q^\star=Q$) and $R$
is skew-Hermitian. We can then replace $P+\calM_n$ with $Q+\calM_n$, and now
$X \mapsto X^\star QX$ takes all its values in $\F$, so $\forall X\in \D^n, \; e(X^\star QX)=2\,X^\star QX$.
Moreover, it is known by polarizing that $X \mapsto X^\star QX$ determines $Q$.
Hence the equivalence class of $P+\calM_n$ determines $Q$ up to star-congruence and multiplication with an element of $\F^\times$.
As a consequence, optimal unital $\F$-affine subspaces of $\Mat_n(\D)$ are classified, up to equivalence, by lists
$([Q_1],\dots,[Q_p])$ of classes of nonisotropic \emph{Hermitian} matrices over $\D$, under the relation of homo-star-congruence.
In particular, if $\D \in \{\C,\H\}$ and $\F=\R$, then we recover from the classification of Hermitian forms that
every optimal unital $\R$-affine subspace of $\Mat_n(\D)$ is equivalent to
$I_n+(\calSH_{n_1}(\D) \vee \cdots \vee \calSH_{n_p}(\D))$ for a unique partition $n=n_1+\cdots+n_p$,
and conversely $I_n+(\calSH_{n_1}(\D) \vee \cdots \vee \calSH_{n_p}(\D))$ is an optimal unital $\R$-affine subspace of $\Mat_n(\D)$
for every partition  $n=n_1+\cdots+n_p$.

For fields of characteristic $2$, Hermitian matrices are skew-Hermitian and hence there is no obvious way to simplify our vision
on the classification of the forms $X \mapsto e(X^\star PX)$. We leave the problem for future work.

\end{document}